\documentclass{article}
\usepackage{amsmath,amsthm,amssymb,fullpage,tikz,mathtools,hyperref,tikz-cd,stmaryrd,float}
\newcommand{\f}[2]{\frac{#1}{#2}}
\newcommand{\pq}[1]{\left(#1\right)}
\newcommand{\abs}[1]{\left\lvert{#1}\right\rvert}

\DeclareRobustCommand{\Sone}{\genfrac[]{0pt}{}}
\DeclareRobustCommand{\Stwo}{\genfrac\{\}{0pt}{}}

\newtheorem{theorem}{Theorem}
\newtheorem{lemma}[theorem]{Lemma}
\newtheorem{proposition}[theorem]{Proposition}
\newtheorem{corollary}[theorem]{Corollary}
\newtheorem{conjecture}[theorem]{Conjecture}
\theoremstyle{definition}
\newtheorem{definition}[theorem]{Definition}
\newtheorem{example}[theorem]{Example}
\numberwithin{theorem}{section}
\title{Counting Parabolic Double Cosets in Symmetric Groups}
\author{Thomas Browning}
\date{August 2021}
\begin{document}
\maketitle
\begin{abstract}
   Billey, Konvalinka, Petersen, Solfstra, and Tenner recently presented a method for counting parabolic double cosets in Coxeter groups, and used it to compute $p_n$, the number of parabolic double cosets in $S_n$, for $n\leq13$.
   In this paper, we derive a new formula for $p_n$ and an efficient polynomial time algorithm for evaluating this formula.
   We use these results to compute $p_n$ for $n\leq5000$ and to prove an asymptotic formula for $p_n$ that was conjectured by Billey et al.
\end{abstract}
\tableofcontents
\section{Introduction}
For a Coxeter system $(W,S)$, each subset $I\subseteq S$ generates a subgroup $W_I$ of $W$.
These subgroups of $W$ are called parabolic subgroups.
A parabolic double coset in the Coxeter system $(W,S)$ is a double coset of the form $W_IwW_J$ for $w\in W$ and $I,J\subseteq S$.
Parabolic double cosets have several properties that make them interesting objects to study.
For example, the parabolic double cosets of a finite Coxeter group $(W,S)$ are rank-symmetric intervals in the Bruhat order on $W$ \cite{kobayashi}, and the poset of presentations of parabolic double cosets gives a boolean complex that retains many properties of the Coxeter complex \cite{peterson}.

The problem of counting parabolic double cosets is considered in \cite{billey}.
Counting parabolic double cosets is hard because a given parabolic double coset $C$ might arise from multiple choices of $I$ and $J$.
To avoid this problem, Billey, Konvalinka, Petersen, Slofstra, and Tenner take the approach of counting lex-minimal presentations \cite{billey}.
The lex-minimal presentation of a parabolic double coset $C$ is the unique presentation $C=W_IwW_J$ such that $w$ is the minimal element of $C$ in the Bruhat order and such that $(\abs{I},\abs{J})$ is lexicographically minimal among all presentations of $C$.
Theorem 4.26 of \cite{billey} gives a formula for the number of parabolic double cosets with a given minimal element $w\in W$, for any Coxeter group $W$.
In the case of the symmetric group, $W=S_n$, summing this formula over all $n!$ elements of $S_n$ will compute $p_n$, the total number of distinct parabolic double cosets in $S_n$.
The sequence $\{p_n\}$ may be found at \cite[A260700]{oeis}.
The values of $p_n$ for $n\leq13$ have been computed using this method \cite{billey}.
However, the summation over all $n!$ elements of $S_n$ limits the number of terms of $\{p_n\}$ that can be computed with this approach.

In this paper we restrict our attention to counting parabolic double cosets in $S_n$.
Two ways of depicting parabolic double cosets in $S_n$ are the balls-in-boxes model considered in \cite{billey} and two-way contingency tables, which are matrices of nonnegative integers with nonzero row sums and column sums.
Diaconis and Gangolli use the balls-in-boxes model to give a bijection between parabolic double cosets in the double quotient $W_I\backslash S_n/W_J=\{W_IwW_J:w\in S_n\}$ and two-way contingency tables with prescribed row sums and column sums \cite{diaconis}.
They attribute the idea behind the bijection to Nantel Bergeron.
This bijection is also used in \cite{peterson}, where it provides an alternate construction of the two-sided Coxeter complex of $S_n$ in terms of two-way contingency tables.

There are three main results of this paper.
The first result is the explicit formula
\begin{equation}
    \label{eq:pnform}
    p_n=\frac{1}{n!}\sum_{m=0}^n\Sone{n}{m}\sum_{c=0}^{\lfloor m/2\rfloor}\sum_{t=2c}^m\binom{m}{t}\pq{\sum_{k=c}^{t-c}(-1)^k(c+k)!\Stwo{t}{c+k}\binom{k-1}{k-c}}\pq{\sum_{j=0}^c\frac{f_{m-t+c,j}f_{m-t+c,c-j}}{j!\,(c-j)!}},
\end{equation}
where the values $f_{n,k}=\sum_{l=k}^nl!\Stwo{n-k}{l-k}$ are generalizations of the Fubini numbers.
The second result is an efficient polynomial time algorithm to evaluate this formula, which is used to compute the values of $p_n$ for $n\leq5000$.
These values of $p_n$ are summarized in the appendix.
The third result is an asymptotic formula for $p_n$ that proves the following conjecture.
\begin{conjecture}[Conjecture 1.4 in \cite{billey}]
\label{conj:asymp}
There exists a constant $K$ so that
\[\f{p_n}{n!}\sim\f{K}{(\log2)^{2n}}.\]
\end{conjecture}
\begin{theorem}
\label{thm:pnthm}
Conjecture \ref{conj:asymp} holds with constant
\[K=\frac{e^{-(\log2)^2/2}}{4(\log2)^2}\approx0.409223.\]
\end{theorem}
We conclude this introduction by outlining the two key ideas behind formula (\ref{eq:pnform}).
The first idea is to associate each parabolic double coset with a canonical two-way contingency table.
In particular, there is a bijection between parabolic double cosets in $S_n$ and two-way contingency tables with sum $n$ that are ``maximal'' (as defined at the end of section 2).
This reduces the problem of computing $p_n$ to the problem of counting maximal two-way contingency tables with sum $n$.

Two-way contingency tables have both a restriction on the rows (nonzero row sums) and a restriction on the columns (nonzero column sums).
These two restrictions are entangled since any nonzero entry in the table both makes its row sum nonzero and its column sum nonzero.
The second idea is to break this entanglement by transforming the problem from counting two-way contingency tables to counting pairs of weak orders (a weak order is a binary relation that is transitive and has no incomparable pairs of elements).

This transformation is applied in sections 3.2 and 3.3 to the problem of counting maximal two-way contingency tables with sum $n$, as well as to the simpler problem of counting all two-way contingency tables with sum $n$.
More generally, this transformation is applicable whenever counting two-way contingency tables with a restriction on the locations of nonzero entries, but not on their specific values.
\section{Background}
In this section, we give an overview of parabolic double cosets in $S_n$, and describe two ways of depicting presentations of parabolic double cosets in $S_n$.
We will use $s_i$ to denote the $i$th adjacent transposition in $S_n$ (the transposition that swaps $i$ and $i+1$).
\begin{definition}
A \textit{parabolic subgroup} $W_I$ in the symmetric group $S_n$ is a subgroup $W_I=\langle I\rangle=\langle s:s\in I\rangle$ of $S_n$ generated by a collection $I\subseteq\{s_1,\ldots,s_{n-1}\}$ of adjacent transpositions in $S_n$.
\end{definition}
\begin{definition}
A \textit{parabolic double coset} $C$ in the symmetric group $S_n$ is a double coset of the form $C=W_IwW_J$ for parabolic subgroups $W_I$ and $W_J$ in $S_n$ and an element $w\in S_n$.
The triple $(I,w,J)$ is called a \textit{presentation} of $C$.
The total number of distinct parabolic double cosets in $S_n$ is denoted by $p_n$.\\
The sequence $\{p_n\}$ may be found at \cite[A260700]{oeis}.
\end{definition}
A parabolic double coset $C$ will usually have many presentations.
For example, if $(I,w,J)$ is a presentation of $C$, then so is $(I,w^\prime,J)$ for any $w^\prime\in C$.
As the next example shows, there could also be multiple choices for the collections $I$ and $J$.
In other words, $C$ might be an element of more than one double quotient
\[W_I\backslash S_n/W_J=\{W_IwW_J:w\in S_n\}.\]
\begin{example}
\label{ex:manypresentations}
The parabolic double coset $C=S_2$ in $S_2$ has six different presentations:
\[\begin{tabular}{c c c}
    $(\{\},12,\{s_1\})$&$(\{s_1\},12,\{s_1\})$&$(\{s_1\},12,\{\})$\\
    $(\{\},21,\{s_1\})$&$(\{s_1\},21,\{s_1\})$&$(\{s_1\},21,\{\})$
\end{tabular}\]
\end{example}
For each element $w\in S_n$, the ball-in-boxes picture of $w$ is given by placing $n$ balls in an $n\times n$ grid of boxes in the layout of the permutation matrix of $w$.
A ball is placed in column $a$ and row $b$ if and only if $w(a)=b$.
We label the columns left-to-right and the rows bottom-to-top, as in the Cartesian coordinate system.
Left-multiplication acts by permuting the rows of the grid and right-multiplication acts by permuting the columns of the grid.

Consider a presentation $(I,w,J)$ of a parabolic double coset $C$ in $S_n$.
The group $W_I$ acts on the balls-in-boxes picture of $w$ by permuting certain rows of the grid.
The group $W_J$ acts on the balls-in-boxes picture of $w$ by permuting certain columns of the grid.
For each adjacent transposition $s_i\notin I$, we draw a horizontal wall between row $i$ and row $i+1$.
For each adjacent transposition $s_j\notin J$, we draw a vertical wall between column $j$ and column $j+1$.
Then $W_I$ acts on the balls-in-boxes picture of $w$ by permuting rows of the grid within the horizontal walls, and $W_J$ acts on the balls-in-boxes picture of $w$ by permuting columns of the grid within the vertical walls.
The resulting picture is called the balls-in-boxes picture with walls for the presentation $(I,w,J)$ \cite{billey}.
\begin{example}[Example 3.5 in \cite{billey}]\label{ex:ballsinboxes}
Let $w=3512467\in S_7$, let $I=\{s_1,s_3,s_4\}$, let $J=\{s_2,s_3,s_5,s_6\}$, and let $C=W_IwW_J$.
Figure 1 depicts the balls-in-boxes picture with walls for the presentation $(I,w,J)$ of $C$.
\begin{figure}[H]
\centering
\begin{tikzpicture}[scale=0.5]
\draw[ultra thick](0,0)--(0,7);
\draw[ultra thick](1,0)--(1,7);
\draw(2,0)--(2,7);
\draw(3,0)--(3,7);
\draw[ultra thick](4,0)--(4,7);
\draw(5,0)--(5,7);
\draw(6,0)--(6,7);
\draw[ultra thick](7,0)--(7,7);
\draw[ultra thick](0,0)--(7,0);
\draw(0,1)--(7,1);
\draw[ultra thick](0,2)--(7,2);
\draw(0,3)--(7,3);
\draw(0,4)--(7,4);
\draw[ultra thick](0,5)--(7,5);
\draw[ultra thick](0,6)--(7,6);
\draw[ultra thick](0,7)--(7,7);
\filldraw(0.5,2.5)circle[radius=0.2];
\filldraw(1.5,4.5)circle[radius=0.2];
\filldraw(2.5,0.5)circle[radius=0.2];
\filldraw(3.5,1.5)circle[radius=0.2];
\filldraw(4.5,3.5)circle[radius=0.2];
\filldraw(5.5,5.5)circle[radius=0.2];
\filldraw(6.5,6.5)circle[radius=0.2];
\draw(-0.5,1)node{\large$s_1$};
\draw(-0.5,3)node{\large$s_3$};
\draw(-0.5,4)node{\large$s_4$};
\draw(2,-0.5)node{\large$s_2$};
\draw(3,-0.5)node{\large$s_3$};
\draw(5,-0.5)node{\large$s_5$};
\draw(6,-0.5)node{\large$s_6$};
\end{tikzpicture}
\caption{A balls-in-boxes picture with walls.}
\end{figure}
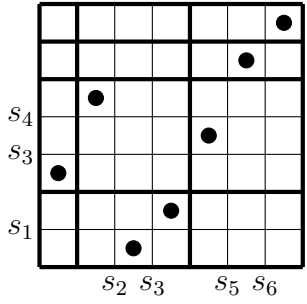
\end{example}
The walls naturally divide the $n\times n$ grid into larger rectangular \textit{cells}.
Notice that the number of balls in each cell remains unchanged under the actions of $W_I$ and $W_J$.
Thus the number of balls in a given cell is constant over all elements of $C$.
By counting the number of balls in each cell, we obtain a matrix of nonnegative integers such that the sum of all of the entries is $n$ and such that each row sum and column sum is strictly positive.
\begin{definition}
\label{def:twoway}
A \textit{two-way contingency table} is a matrix of nonnegative integers such that each row sum and column sum is strictly positive.
Let $p_{n,0}$ denote the number of two-way contingency tables with sum $n$.\\
The sequence $\{p_{n,0}\}$ may be found at \cite[A120733]{oeis}.
\end{definition}
Thus, each presentation $(I,w,J)$ of $C$ gives rise to a two-way contingency table with sum $n$.
Conversely, given a two-way contingency table with sum $n$, it is possible to recover the collections $I$ and $J$, as well as the parabolic double coset $C\in W_I\backslash S_n/W_J$ \cite[Section 3B]{diaconis}.
As a consequence, two-way contingency tables with sum $n$ are in bijection with triples $(C,I,J)$ consisting of two collections $I,J\subseteq\{s_1,\ldots,s_{n-1}\}$ and a parabolic double coset $C\in W_I\backslash S_n/W_J$ \cite{peterson}.
Another way to put this is that $p_{n,0}=\sum_I\sum_J\abs{W_I\backslash S_n/W_J}$, because the right hand side counts triples $(C,I,J)$.
\begin{example}
\label{ex:contingency}
The two-way contingency table associated to the triple $(C,I,J)$ from Example \ref{ex:ballsinboxes} is the matrix
\[\begin{pmatrix}0&0&1\\0&0&1\\1&1&1\\0&2&0\end{pmatrix}.\]
\end{example}
Even though a given parabolic double coset might have multiple distinct two-way contingency tables arising from multiple choices for the collections $I$ and $J$, the following proposition shows that it is always possible to choose $I$ and $J$ to be simultaneously maximal, resulting in a canonical two-way contingency table associated to each parabolic double coset.
\begin{proposition}
\label{prop:maximalpresentation}
Let $C$ be a parabolic double coset in $S_n$.
Let $I=\{s_i:s_iC=C\}$, let $J=\{s_i:Cs_i=C\}$, and let $w\in C$ be arbitrary.
Then $(I,w,J)$ is a presentation of $C$ and is maximal in the sense that if $(I^\prime,w^\prime,J^\prime)$ is any other presentation of $C$, then $I^\prime\subseteq I$ and $J^\prime\subseteq J$.
\end{proposition}
\begin{proof}
Consider the stabilizer subgroups $H_L=\{\sigma\in S_n\colon\sigma C=C\}$ and $H_R=\{\sigma\in S_n\colon C\sigma=C\}$.
Then $I=H_L\cap\{s_1,\ldots,s_{n-1}\}$ and $J=H_R\cap\{s_1,\ldots,s_{n-1}\}$.
In particular, $I\subseteq H_L$ and $J\subseteq H_R$ so $W_I=\langle I\rangle\subseteq H_L$ and $W_J=\langle J\rangle\subseteq H_R$.
Since $w\in C$, this shows that $W_IwW_J\subseteq C$.
We will now show that $C\subseteq W_IwW_J$.

Let $(I^\prime,w^\prime,J^\prime)$ be any presentation of $C$.
Then $C=W_{I^\prime}w^\prime W_{J^\prime}$.
If $\sigma\in I^\prime$, then
\[\sigma C=(\sigma W_{I^\prime})(w^\prime W_{J^\prime})=(W_{I^\prime})(w^\prime W_{J^\prime})=C\]
so $\sigma\in H_L$.
Thus, $I^\prime\subseteq H_L$ and similarly $J^\prime\subseteq H_R$.
Intersecting with $\{s_1,\ldots,s_{n-1}\}$ shows that $I^\prime\subseteq I$ and $J^\prime\subseteq J$.
Then $W_{I^\prime}\subseteq W_I$ and $W_{J^\prime}\subseteq W_J$ so
\[C=W_{I^\prime}w^\prime W_{J^\prime}=W_{I^\prime}wW_{J^\prime}\subseteq W_IwW_J.\]
Combining this with the earlier inclusion $W_IwW_J\subseteq C$ gives the equality $W_IwW_J=C$.
Thus, $(I,w,J)$ is a presentation of $C$.
Finally, $(I,w,J)$ is maximal since we already showed that if $(I^\prime,w^\prime,J^\prime)$ is any presentation of $C$, then $I^\prime\subseteq I$ and $J^\prime\subseteq J$.
\end{proof}
For an alternative proof of Proposition \ref{prop:maximalpresentation}, see Proposition 3.7 in \cite{billey}.
Proposition \ref{prop:maximalpresentation} gives a bijection
\begin{align*}
    \left\{\substack{\text{\small{parabolic double}}\\\text{\small{cosets $C\in S_n$}}}\right\}&\longleftrightarrow\left\{\substack{\text{\small{Triples $(C,I,J)$ consisting of two collections}}\\\text{\small{$I,J\subseteq\{s_1,\ldots,s_{n-1}\}$ and a parabolic}}\\\text{\small{double coset $C\in W_I\backslash S_n/W_J$ such that }}\\\text{\small{$I=\{s_i:s_iC=C\}$ and $J=\{s_i:Cs_i=C\}$}}}\right\}.
\end{align*}
This bijection is given by
\begin{align*}
    C&\longmapsto(C,\{s_i:s_iC=C\},\{s_i:Cs_i=C\}),\\
    C&\longmapsfrom(C,I,J),
\end{align*}
where Proposition \ref{prop:maximalpresentation} guarentees that the triple $(C,\{s_i:s_iC=C\},\{s_i:Cs_i=C\})$ satisfies the condition $C\in W_I\backslash S_n/W_J$ required by the set on the right hand side of the bijection.

The bijection between triples $(C,I,J)$ and two-way contingency tables described after Definition \ref{def:twoway} gives a bijection
\[\left\{\substack{\text{\small{Triples $(C,I,J)$ consisting of two collections}}\\\text{\small{$I,J\subseteq\{s_1,\ldots,s_{n-1}\}$ and a parabolic}}\\\text{\small{double coset $C\in W_I\backslash S_n/W_J$ such that }}\\\text{\small{$I=\{s_i:s_iC=C\}$ and $J=\{s_i:Cs_i=C\}$}}}\right\}\longleftrightarrow\left\{\substack{\text{\small{Two-way contingency tables with sum $n$}}\\\text{\small{whose associated triple $(C,I,J)$ satisfies}}\\\text{\small{$I=\{s_i:s_iC=C\}$ and $J=\{s_i:Cs_i=C\}$}}}\right\}.\]
This motivates the following definition.
\begin{definition}
A two-way contingency table is said to be \textit{maximal} if the associated triple $(C,I,J)$ satisfies $I=\{s_i:s_iC=C\}$ and $J=\{s_i:Cs_i=C\}$.
\end{definition}
We remark that the maximal two-way contingency table associated to a given parabolic double coset has the smallest possible dimensions and the largest possible entries.

Putting this all together gives a bijection
\[\left\{\substack{\text{\small{parabolic double}}\\\text{\small{cosets $C\in S_n$}}}\right\}\longleftrightarrow\left\{\substack{\text{\small{Maximal two-way}}\\\text{\small{contingency tables}}\\\text{\small{with sum $n$}}}\right\}.\]
\begin{corollary}
\label{cor:pncombo}
The sequence $\{p_{n,0}\}$ counts the number of two-way contingency tables with sum $n$.
The sequence $\{p_n\}$ counts the number of maximal two-way contingency tables with sum $n$.
\end{corollary}
We will now turn our attention to counting maximal two-way contingency tables with sum $n$.
\section{Computation of \texorpdfstring{$p_n$}{Lg}}
\subsection{Characterization of Maximality}
The following proposition describes how to determine whether the conditions $s_iC=C$ and $Cs_i=C$ are satisfied from the balls-in-boxes picture with walls.
We will use the term \textit{cell} to mean one of the rectangular regions formed by the walls of the balls-in-boxes picture with walls.
\begin{proposition}
\label{prop:contingencyclosure}
Let $C=W_IwW_J$ be a parabolic double coset in $S_n$, and let $P$ be the balls-in-boxes picture with walls for the presentation $(I,w,J)$.
\begin{enumerate}
    \item Let $s_i\in\{s_1,\ldots,s_{n-1}\}\setminus I$.
    Then $s_iC=C$ if and only if the row of cells of $P$ directly below the $s_i$-wall and the row of cells of $P$ directly above the $s_i$-wall each have only one nonempty cell, the two of which are vertically adjacent.
    \item Let $s_j\in\{s_1,\ldots,s_{n-1}\}\setminus J$.
    Then $Cs_j=C$ if and only if the column of cells of $P$ directly left the $s_j$-wall and the column of cells of $P$ directly right the $s_j$-wall each have only one nonempty cell, the two of which are horizontally adjacent.
\end{enumerate}
\end{proposition}
\begin{proof}
Let $s_i\in\{s_1,\ldots,s_{n-1}\}\setminus I$.
First suppose that the row of cells of $P$ directly below the $s_i$-wall and the row of cells of $P$ directly above the $s_i$-wall each have only one nonempty cell, the two of which are vertically adjacent.
Let $A$ be the nonempty cell of $P$ directly below the $s_i$-wall and let $B$ be the nonempty cell of $P$ directly above the $s_i$-wall.
Since $A$ and $B$ are each the only nonempty cells in their row of cells, the ball in row $i$ must lie in $A$ and the ball in row $i+1$ must lie in $B$.
Then the ball in column $w^{-1}(i)$ must lie in $A$ and the ball in column $w^{-1}(i+1)$ must lie in $B$.
Since the cells $A$ and $B$ are vertically adjacent, $w^{-1}(i)$ and $w^{-1}(i+1)$ lie in the same column of cells.
Then the transposition $w^{-1}s_iw$ that swaps $w^{-1}(i)$ and $w^{-1}(i+1)$ is an element of $W_J$.
Taking $w^{-1}s_iw\in W_J$ and multiplying on the left by $w$ gives $s_iw\in wW_J\subseteq C$.
Since $w\in C$ was arbitrary, this shows that $s_iC\subseteq C$.
Since $\abs{s_iC}=\abs{C}$, we must have $s_iC=C$.

For the converse, suppose that $P$ has a nonempty cell $A$ directly below the $s_i$-wall and a nonempty cell $B$ directly above the $s_i$-wall such that $A$ and $B$ are not vertically adjacent.
This is just the negation of the original assumption on $P$.
Since $A$ and $B$ are nonempty, we can multiply $w$ on the left by an element $x\in W_I$ to make $A$ contain the ball in row $i$ and to make $B$ contain the ball in row $i+1$.
Then $s_ixw$ has fewer balls in cells $A$ and $B$ than $xw$.
However, the number of balls in a given cell is constant over all elements of $C$.
Since $xw\in C$, this forces $s_ixw\not\in C$.
In particular, $s_iC\neq C$.
This proves the first statement.
The proof of the second statement is similar.
\end{proof}
By definition, a two-way contingency table is maximal precisely when neither of the ``if and only if''s in Proposition \ref{prop:contingencyclosure} are ever satisfied.
This gives a characterization of maximal two-way contingency tables.
\begin{corollary}
\label{cor:maximaltable1}
Let $T$ be a two-way contingency table.
Then $T$ is maximal if and only if $T$ satisfies both of the following two conditions:
\begin{enumerate}
    \item There does not exist a pair of vertically adjacent nonzero entries, each of which is the only nonzero entry in its row.
    \item There does not exist a pair of horizontally adjacent nonzero entries, each of which is the only nonzero entry in its column.
\end{enumerate}
\end{corollary}
\begin{example}
There is $p_1=1$ maximal two-way contingency table with sum 1:
\[\begin{tikzpicture}[scale=0.5]
\draw[ultra thick](0,0)--(0,1)--(1,1)--(1,0)--(0,0);
\draw(0.5,0.5)node{\large$1$};
\end{tikzpicture}\]
There are $p_2=3$ maximal two-way contingency tables with sum 2:
\[\begin{tikzpicture}[scale=0.5]
\draw[ultra thick](10,0.5)--(10,1.5)--(11,1.5)--(11,0.5)--(10,0.5);
\draw(10.5,1)node{\large$2$};

\draw[ultra thick](3,0)--(3,2)--(5,2)--(5,0)--(3,0);
\draw(4,0)--(4,2);
\draw(3,1)--(5,1);
\draw(3.5,0.5)node{\large$1$};
\draw(3.5,1.5)node{\large$0$};
\draw(4.5,1.5)node{\large$1$};
\draw(4.5,0.5)node{\large$0$};

\draw[ultra thick](6.5,0)--(6.5,2)--(8.5,2)--(8.5,0)--(6.5,0);
\draw(7.5,0)--(7.5,2);
\draw(6.5,1)--(8.5,1);
\draw(7,0.5)node{\large$0$};
\draw(7,1.5)node{\large$1$};
\draw(8,1.5)node{\large$0$};
\draw(8,0.5)node{\large$1$};
\end{tikzpicture}\]
There are $p_3=19$ maximal two-way contingency tables with sum 3:
\[\begin{tikzpicture}[scale=0.5]
\draw[ultra thick](3,0)--(3,2)--(5,2)--(5,0)--(3,0);
\draw(4,0)--(4,2);
\draw(3,1)--(5,1);
\draw(3.5,1.5)node{\large$1$};
\draw(3.5,0.5)node{\large$0$};
\draw(4.5,0.5)node{\large$1$};
\draw(4.5,1.5)node{\large$1$};

\draw[ultra thick](6.5,0)--(6.5,2)--(8.5,2)--(8.5,0)--(6.5,0);
\draw(7.5,0)--(7.5,2);
\draw(6.5,1)--(8.5,1);
\draw(7,1.5)node{\large$0$};
\draw(7,0.5)node{\large$1$};
\draw(8,0.5)node{\large$1$};
\draw(8,1.5)node{\large$1$};

\draw[ultra thick](10,0)--(10,2)--(12,2)--(12,0)--(10,0);
\draw(11,0)--(11,2);
\draw(10,1)--(12,1);
\draw(10.5,1.5)node{\large$1$};
\draw(10.5,0.5)node{\large$1$};
\draw(11.5,0.5)node{\large$1$};
\draw(11.5,1.5)node{\large$0$};

\draw[ultra thick](13.5,0)--(13.5,2)--(15.5,2)--(15.5,0)--(13.5,0);
\draw(14.5,0)--(14.5,2);
\draw(13.5,1)--(15.5,1);
\draw(14,1.5)node{\large$1$};
\draw(14,0.5)node{\large$1$};
\draw(15,0.5)node{\large$0$};
\draw(15,1.5)node{\large$1$};

\draw[ultra thick](3,3.5)--(3,5.5)--(5,5.5)--(5,3.5)--(3,3.5);
\draw(4,3.5)--(4,5.5);
\draw(3,4.5)--(5,4.5);
\draw(3.5,5)node{\large$0$};
\draw(3.5,4)node{\large$1$};
\draw(4.5,4)node{\large$0$};
\draw(4.5,5)node{\large$2$};

\draw[ultra thick](6.5,3.5)--(6.5,5.5)--(8.5,5.5)--(8.5,3.5)--(6.5,3.5);
\draw(7.5,3.5)--(7.5,5.5);
\draw(6.5,4.5)--(8.5,4.5);
\draw(7,5)node{\large$1$};
\draw(7,4)node{\large$0$};
\draw(8,4)node{\large$2$};
\draw(8,5)node{\large$0$};

\draw[ultra thick](10,3.5)--(10,5.5)--(12,5.5)--(12,3.5)--(10,3.5);
\draw(11,3.5)--(11,5.5);
\draw(10,4.5)--(12,4.5);
\draw(10.5,5)node{\large$0$};
\draw(10.5,4)node{\large$2$};
\draw(11.5,4)node{\large$0$};
\draw(11.5,5)node{\large$1$};

\draw[ultra thick](13.5,3.5)--(13.5,5.5)--(15.5,5.5)--(15.5,3.5)--(13.5,3.5);
\draw(14.5,3.5)--(14.5,5.5);
\draw(13.5,4.5)--(15.5,4.5);
\draw(14,5)node{\large$2$};
\draw(14,4)node{\large$0$};
\draw(15,4)node{\large$1$};
\draw(15,5)node{\large$0$};

\draw[ultra thick](17,0)--(17,2)--(20,2)--(20,0)--(17,0);
\draw(18,0)--(18,2);
\draw(19,0)--(19,2);
\draw(17,1)--(20,1);
\draw(17.5,1.5)node{\large$0$};
\draw(18.5,1.5)node{\large$1$};
\draw(19.5,1.5)node{\large$0$};
\draw(17.5,0.5)node{\large$1$};
\draw(18.5,0.5)node{\large$0$};
\draw(19.5,0.5)node{\large$1$};

\draw[ultra thick](17,3.5)--(17,5.5)--(20,5.5)--(20,3.5)--(17,3.5);
\draw(18,3.5)--(18,5.5);
\draw(19,3.5)--(19,5.5);
\draw(17,4.5)--(20,4.5);
\draw(17.5,5)node{\large$1$};
\draw(18.5,5)node{\large$0$};
\draw(19.5,5)node{\large$1$};
\draw(17.5,4)node{\large$0$};
\draw(18.5,4)node{\large$1$};
\draw(19.5,4)node{\large$0$};

\draw[ultra thick](21.5,0)--(21.5,3)--(23.5,3)--(23.5,0)--(21.5,0);
\draw(22.5,0)--(22.5,3);
\draw(21.5,1)--(23.5,1);
\draw(21.5,2)--(23.5,2);
\draw(22,2.5)node{\large$1$};
\draw(23,2.5)node{\large$0$};
\draw(22,1.5)node{\large$0$};
\draw(23,1.5)node{\large$1$};
\draw(22,0.5)node{\large$1$};
\draw(23,0.5)node{\large$0$};

\draw[ultra thick](25,0)--(25,3)--(27,3)--(27,0)--(25,0);
\draw(26,0)--(26,3);
\draw(25,1)--(27,1);
\draw(25,2)--(27,2);
\draw(25.5,2.5)node{\large$0$};
\draw(26.5,2.5)node{\large$1$};
\draw(25.5,1.5)node{\large$1$};
\draw(26.5,1.5)node{\large$0$};
\draw(25.5,0.5)node{\large$0$};
\draw(26.5,0.5)node{\large$1$};

\draw[ultra thick](21.5,5.5)--(21.5,4.5)--(22.5,4.5)--(22.5,5.5)--(21.5,5.5);
\draw(22,5)node{\large$3$};

\draw[ultra thick](3,-1.5)--(6,-1.5)--(6,-4.5)--(3,-4.5)--(3,-1.5);
\draw(3,-2.5)--(6,-2.5);
\draw(3,-3.5)--(6,-3.5);
\draw(4,-1.5)--(4,-4.5);
\draw(5,-1.5)--(5,-4.5);
\draw(3.5,-2)node{\large{$0$}};
\draw(4.5,-2)node{\large{$0$}};
\draw(5.5,-2)node{\large{$1$}};
\draw(3.5,-3)node{\large{$0$}};
\draw(4.5,-3)node{\large{$1$}};
\draw(5.5,-3)node{\large{$0$}};
\draw(3.5,-4)node{\large{$1$}};
\draw(4.5,-4)node{\large{$0$}};
\draw(5.5,-4)node{\large{$0$}};

\draw[ultra thick](7.5,-1.5)--(10.5,-1.5)--(10.5,-4.5)--(7.5,-4.5)--(7.5,-1.5);
\draw(7.5,-2.5)--(10.5,-2.5);
\draw(7.5,-3.5)--(10.5,-3.5);
\draw(8.5,-1.5)--(8.5,-4.5);
\draw(9.5,-1.5)--(9.5,-4.5);
\draw(8,-2)node{\large{$0$}};
\draw(9,-2)node{\large{$1$}};
\draw(10,-2)node{\large{$0$}};
\draw(8,-3)node{\large{$0$}};
\draw(9,-3)node{\large{$0$}};
\draw(10,-3)node{\large{$1$}};
\draw(8,-4)node{\large{$1$}};
\draw(9,-4)node{\large{$0$}};
\draw(10,-4)node{\large{$0$}};

\draw[ultra thick](12,-1.5)--(15,-1.5)--(15,-4.5)--(12,-4.5)--(12,-1.5);
\draw(12,-2.5)--(15,-2.5);
\draw(12,-3.5)--(15,-3.5);
\draw(13,-1.5)--(13,-4.5);
\draw(14,-1.5)--(14,-4.5);
\draw(12.5,-2)node{\large{$0$}};
\draw(13.5,-2)node{\large{$0$}};
\draw(14.5,-2)node{\large{$1$}};
\draw(12.5,-3)node{\large{$1$}};
\draw(13.5,-3)node{\large{$0$}};
\draw(14.5,-3)node{\large{$0$}};
\draw(12.5,-4)node{\large{$0$}};
\draw(13.5,-4)node{\large{$1$}};
\draw(14.5,-4)node{\large{$0$}};

\draw[ultra thick](16.5,-1.5)--(19.5,-1.5)--(19.5,-4.5)--(16.5,-4.5)--(16.5,-1.5);
\draw(16.5,-2.5)--(19.5,-2.5);
\draw(16.5,-3.5)--(19.5,-3.5);
\draw(17.5,-1.5)--(17.5,-4.5);
\draw(18.5,-1.5)--(18.5,-4.5);
\draw(17,-2)node{\large{$0$}};
\draw(18,-2)node{\large{$1$}};
\draw(19,-2)node{\large{$0$}};
\draw(17,-3)node{\large{$1$}};
\draw(18,-3)node{\large{$0$}};
\draw(19,-3)node{\large{$0$}};
\draw(17,-4)node{\large{$0$}};
\draw(18,-4)node{\large{$0$}};
\draw(19,-4)node{\large{$1$}};

\draw[ultra thick](21,-1.5)--(24,-1.5)--(24,-4.5)--(21,-4.5)--(21,-1.5);
\draw(21,-2.5)--(24,-2.5);
\draw(21,-3.5)--(24,-3.5);
\draw(22,-1.5)--(22,-4.5);
\draw(23,-1.5)--(23,-4.5);
\draw(21.5,-2)node{\large{$1$}};
\draw(22.5,-2)node{\large{$0$}};
\draw(23.5,-2)node{\large{$0$}};
\draw(21.5,-3)node{\large{$0$}};
\draw(22.5,-3)node{\large{$0$}};
\draw(23.5,-3)node{\large{$1$}};
\draw(21.5,-4)node{\large{$0$}};
\draw(22.5,-4)node{\large{$1$}};
\draw(23.5,-4)node{\large{$0$}};

\draw[ultra thick](25.5,-1.5)--(28.5,-1.5)--(28.5,-4.5)--(25.5,-4.5)--(25.5,-1.5);
\draw(25.5,-2.5)--(28.5,-2.5);
\draw(25.5,-3.5)--(28.5,-3.5);
\draw(26.5,-1.5)--(26.5,-4.5);
\draw(27.5,-1.5)--(27.5,-4.5);
\draw(26,-2)node{\large{$1$}};
\draw(27,-2)node{\large{$0$}};
\draw(28,-2)node{\large{$0$}};
\draw(26,-3)node{\large{$0$}};
\draw(27,-3)node{\large{$1$}};
\draw(28,-3)node{\large{$0$}};
\draw(26,-4)node{\large{$0$}};
\draw(27,-4)node{\large{$0$}};
\draw(28,-4)node{\large{$1$}};
\end{tikzpicture}\]
\end{example}
\subsection{Sequence Transformation}
Note that the conditions in Corollary \ref{cor:maximaltable1} only depend on the configuration of nonzero entries of $T$.
Because of this, it will be helpful to ignore the specific values of the nonzero entries of $T$.
\begin{definition}
A \textit{pattern} is a two-way contingency table consisting of 0's and 1's.
\begin{enumerate}
    \item Let $\mathcal P_{m,0}$ denote the collection of all patterns with sum $m$.
    \item Let $\mathcal P_m$ denote the collection of all patterns with sum $m$ that satisfy the conditions of Corollary \ref{cor:maximaltable1}.
    \item For a two-way contingency table $T$, we define the pattern associated to $T$ to be the two-way contingency table formed by replacing all nonzero entries of $T$ with 1's.
\end{enumerate}
\end{definition}
For each pattern $P$ with sum $m$, there are exactly $\binom{n-1}{n-m}$ two-way contingency tables with sum $n$ whose associated pattern is $P$.
This is because two-way contingency tables with sum $n$ whose associated pattern is $P$ are in bijection with ways to surjectively distribute $n$ identical balls among $m$ distinguishable urns.

Then combining Corollary \ref{cor:pncombo} with Corollary \ref{cor:maximaltable1} gives the formulas
\[p_{n,0}=\sum_{m=0}^n\binom{n-1}{n-m}\abs{\mathcal P_{m,0}}\quad\text{and}\quad p_n=\sum_{m=0}^n\binom{n-1}{n-m}\abs{\mathcal P_m}.\]
Here we use the conventions $\binom{-1}{0}=1$ and also $\binom{k-1}{k}=0$ for $k\geq1$, since $\binom{n-1}{n-m}$ is counting ways to surjectively distribute $n$ identical balls among $m$ distinguishable urns.
We now invoke the identity
\[\binom{n-1}{n-m}=\f{m!}{n!}\sum_{k=m}^n\Sone{n}{k}\Stwo{k}{m},\]
where $\Sone{\cdot}{\cdot}$ denotes unsigned Stirling numbers of the first kind and $\Stwo{\cdot}{\cdot}$ denotes Stirling numbers of the second kind (see the exercise on Lah numbers in \cite{riordan}).
If we define
\[q_{k,0}=\sum_{m=0}^km!\Stwo{k}{m}\abs{\mathcal P_{m,0}}\quad\text{and}\quad q_k=\sum_{m=0}^km!\Stwo{k}{m}\abs{\mathcal P_m},\]
then we obtain the identities
\[p_{n,0}=\f{1}{n!}\sum_{k=0}^n\Sone{n}{k}q_{k,0}\quad\text{and}\quad p_n=\f{1}{n!}\sum_{k=0}^n\Sone{n}{k}q_k.\]
This reduces the computation of the sequences $\{p_{n,0}\}$ and $\{p_n\}$ to the computation of the sequences $\{q_{n,0}\}$ and $\{q_n\}$.
The following proposition gives a combinatorial interpretation of the sequences $\{q_{n,0}\}$ and $\{q_n\}$.
\begin{proposition}
\label{prop:combo_distrib_matrix}
The sequence $\{q_{n,0}\}$ counts the number of functions from $n$ distinguishable balls to a rectangular matrix of boxes, such that each row and column is nonempty.
The sequence $\{q_n\}$ counts the number of such functions that satisfy the following two conditions:
\begin{enumerate}
    \item There does not exist a pair of vertically adjacent nonempty boxes, each of which is the only nonempty box in its row.
    \item There does not exist a pair of horizontally adjacent nonempty boxes, each of which is the only nonempty box in its column.
\end{enumerate}
\end{proposition}
\begin{proof}
Relabeling indices (from $q_{k,0}$ to $q_{n,0}$ and from $q_k$ to $q_n$) gives
\[q_{n,0}=\sum_{k=0}^nk!\Stwo{n}{k}\abs{\mathcal P_{k,0}}\quad\text{and}\quad q_n=\sum_{k=0}^nk!\Stwo{n}{k}\abs{\mathcal P_k}.\]
Then the proposition follows from the fact that $k!\Stwo{n}{k}$ counts the number of surjective functions from $n$ distinguishable balls to the $k$ nonempty boxes prescribed by the pattern in $\mathcal P_{k,0}$ or $\mathcal P_k$.
\end{proof}
\subsection{Pairs of Weak Orders}
\begin{definition}
A \textit{weak order} on a set $S$ is a binary relation $\leq$ on $S$ that is transitive and has no incomparable pairs of elements (so for all $s,t\in S$, $s\leq t$ or $t\leq s$).
The total number of weak orders on $\{1,\ldots,n\}$ is denoted by $f_n$.
The sequence $\{f_n\}$ is known as the \textit{Fubini numbers} or the \textit{ordered Bell numbers}, and may be found at \cite[A000670]{oeis}.
\end{definition}
Given a weak order on a set $S$, it is possible for two elements $s,t\in S$ to be tied with each other, in the sense that both $s\leq t$ and $t\leq s$.
The ``tied'' relation is an equivalence relation and partitions the set $S$ into tied subsets.
Furthermore, the weak order on $S$ gives a total order on this partition.
In other words, we obtain an ordered set-partition $S=S_1\cup\cdots\cup S_k$.
Conversely, an ordered set-partition $S=S_1\cup\cdots\cup S_k$ determines a weak order on $S$ by setting $s\leq t$ whenever $s\in S_i$ and $t\in S_j$ and $i\leq j$.
Thus, there is a bijection between weak orders on $S$ and ordered set-partitions on $S$.
In what follows, we will freely translate back and forth between weak orders and ordered set-partitions.

In the special case where $S=\{1,\ldots,n\}$, weak orders on $\{1,\ldots,n\}$ correspond to ordered set-partitions $\{1,\ldots,n\}=S_1\cup\cdots\cup S_k$.
Summing over the possible values for $k$ gives the formula $f_n=\sum_{k=0}^nk!\Stwo{n}{k}$.
\begin{definition}
Let $((S_1,\ldots, S_k),(T_1,\ldots,T_{k^\prime}))$ be a pair of ordered set-partitions of $\{1,\ldots,n\}$.
A \textit{left consecutive embedding} is a consecutive pair $(i,i+1)_l$ such that $S_i\cup S_{i+1}\subseteq T_j$ for some (necessarily unique) $T_j$.
Similarly, a \textit{right consecutive embedding} is a consecutive pair $(i,i+1)_r$ such that $T_i\cup T_{i+1}\subseteq S_j$ for some (necessarily unique) $S_j$.
The phrase \textit{consecutive embedding} will refer to either a left consecutive embedding or a right consecutive embedding.
\end{definition}
Since weak orders can be viewed as ordered set-partitions, it makes sense to talk about consecutive embeddings in a pair of weak orders.
\begin{proposition}
\label{prop:qncomboweak}
The sequence $\{q_{n,0}\}$ counts the total number of pairs of weak orders on $\{1,\ldots,n\}$, so $q_{n,0}=f_n^2$.
The sequence $\{q_n\}$ counts the number of such pairs that have no consecutive embeddings.
\end{proposition}
\begin{proof}
Consider one of the $q_{n,0}$ functions from $n$ distinguishable balls to a rectangular matrix of boxes, such that each row and column is nonempty.
Considering the row of each ball gives a weak order on $\{1,\ldots,n\}$.
Similarly, considering the column of each ball gives a weak order on $\{1,\ldots,n\}$.
Conversely, given this pair of weak orders on $\{1,\ldots,n\}$, we can recover the dimensions of the matrix and the position of each ball.
In particular, the dimensions of the matrix are equal to the number of parts of the two ordered set-partitions, and the position of the $i$th ball in the matrix is given by the location of $i$ within the two ordered set-partitions.
This gives a bijection
\[\left\{\substack{\text{\small{Functions from $n$ distinguishable balls}}\\\text{\small{to a rectangular matrix of boxes such}}\\\text{\small{that each row and column is nonempty}}}\right\}\longleftrightarrow\left\{\substack{\text{\small{Pairs of weak orders}}\\\text{\small{on $\{1,\ldots,n\}$}}}\right\}.\]
In particular, the sequence $\{q_{n,0}\}$ counts the total number of pairs of weak orders on $\{1,\ldots,n\}$.

As for $\{q_n\}$, observe that this bijection takes failures of conditions 1 and 2 of Corollary \ref{cor:maximaltable1} to left and right consecutive embeddings.
\end{proof}
\begin{example}
There is $q_{1,0}=1$ pair of weak orders on $\{1\}$:
\[\begin{tabular}{|c|c|}\hline\begin{tabular}{c}1\end{tabular}&\begin{tabular}{c}1\end{tabular}\\\hline\end{tabular}\]
This pair of weak orders has no consecutive embeddings, so we also have $q_1=1$.
Here we write the first weak order in the first column, and the second weak order in the second column.
Tied elements of a weak order are written in the same row as each other.

There are $q_{2,0}=9$ pairs of weak orders on $\{1,2\}$:
\[
\begin{tabular}{|c|c|}\hline\begin{tabular}{c}1\\2\end{tabular}&\begin{tabular}{c}1\\2\end{tabular}\\\hline\end{tabular}\quad
\begin{tabular}{|c|c|}\hline\begin{tabular}{c}1\\2\end{tabular}&\begin{tabular}{c}2\\1\end{tabular}\\\hline\end{tabular}\quad
\begin{tabular}{|c|c|}\hline\begin{tabular}{c}2\\1\end{tabular}&\begin{tabular}{c}1\\2\end{tabular}\\\hline\end{tabular}\quad
\begin{tabular}{|c|c|}\hline\begin{tabular}{c}2\\1\end{tabular}&\begin{tabular}{c}2\\1\end{tabular}\\\hline\end{tabular}\quad
\begin{tabular}{|c|c|}\hline\begin{tabular}{c}1,2\end{tabular}&\begin{tabular}{c}1,2\end{tabular}\\\hline\end{tabular}
\]
\[
\begin{tabular}{|c|c|}\hline\begin{tabular}{c}1,2\end{tabular}&\begin{tabular}{c}1\\2\end{tabular}\\\hline\end{tabular}\quad
\begin{tabular}{|c|c|}\hline\begin{tabular}{c}1,2\end{tabular}&\begin{tabular}{c}2\\1\end{tabular}\\\hline\end{tabular}\quad
\begin{tabular}{|c|c|}\hline\begin{tabular}{c}1\\2\end{tabular}&\begin{tabular}{c}1,2\end{tabular}\\\hline\end{tabular}\quad
\begin{tabular}{|c|c|}\hline\begin{tabular}{c}2\\1\end{tabular}&\begin{tabular}{c}1,2\end{tabular}\\\hline\end{tabular}
\]
Of these, the five in the first row have no consecutive embeddings, whereas the first two in the second row each have a right consecutive embedding, and the last two in the second row each have a left consecutive embedding.
Thus, $q_2=5$.
Then
\[p_{2,0}=\frac{1}{2!}\left(\Sone{2}{1}q_{1,0}+\Sone{2}{2}q_{2,0}\right)=\frac{1}{2}(1+9)=5\qquad\text{and}\qquad p_2=\frac{1}{2!}\left(\Sone{2}{1}q_1+\Sone{2}{1}q_2\right)=\frac{1}{2}(1+5)=3.\]
\end{example}
\subsection{Inclusion-Exclusion}
Recall that we have reduced the computation of the sequences $\{p_{n,0}\}$ and $\{p_n\}$ to the computation of the sequences $\{q_{n,0}\}$ and $\{q_n\}$.
The formula $q_{n,0}=f_n^2$ from Proposition \ref{prop:qncomboweak} gives a fast polynomial-time algorithm for computing the sequence $\{q_{n,0}\}$.
We now restrict our attention to computing the sequence $\{q_n\}$.
This subsection will give an inclusion-exclusion formula for $q_n$.
\begin{definition}\ 
\begin{enumerate}
    \item A \textit{pair of weak orders on $\{1,\ldots,n\}$ with distinguished consecutive embeddings} consists of a pair of weak orders $(\leq,\leq^\prime)$ on $\{1,\ldots,n\}$, together with set $\mathcal C$ of consecutive embeddings in $(\leq,\leq^\prime)$.
    In other words, $\mathcal C$ is a subset of the set of all consecutive embeddings in $(\leq,\leq^\prime)$.
    \item A \textit{pairs of weak orders on $\{1,\ldots,n\}$ with $k$ distinguished consecutive embeddings} is a pair of weak orders on $\{1,\ldots,n\}$ with distinguished consecutive embeddings whose set $\mathcal C$ has cardinality $k$.
    \item Let $q_{n,k}$ count the number of pairs of weak orders on $\{1,\ldots,n\}$ with $k$ distinguished consecutive embeddings.
\end{enumerate}
\end{definition}
The notation $q_{n,k}$ does not conflict with the notation $q_{n,0}$ because pairs of weak orders with 0 distinguished consecutive embeddings are in bijection with (ordinary) pairs of weak orders.
\begin{example}
We will circle distinguished consecutive embeddings.
There are $q_{2,0}=9$ pairs of weak orders on $\{1,2\}$ with 0 distinguished consecutive embeddings:
\[
\begin{tabular}{|c|c|}\hline\begin{tabular}{c}1\\2\end{tabular}&\begin{tabular}{c}1\\2\end{tabular}\\\hline\end{tabular}\quad
\begin{tabular}{|c|c|}\hline\begin{tabular}{c}1\\2\end{tabular}&\begin{tabular}{c}2\\1\end{tabular}\\\hline\end{tabular}\quad
\begin{tabular}{|c|c|}\hline\begin{tabular}{c}2\\1\end{tabular}&\begin{tabular}{c}1\\2\end{tabular}\\\hline\end{tabular}\quad
\begin{tabular}{|c|c|}\hline\begin{tabular}{c}2\\1\end{tabular}&\begin{tabular}{c}2\\1\end{tabular}\\\hline\end{tabular}\quad
\begin{tabular}{|c|c|}\hline\begin{tabular}{c}1,2\end{tabular}&\begin{tabular}{c}1,2\end{tabular}\\\hline\end{tabular}
\]
\[
\begin{tabular}{|c|c|}\hline\begin{tabular}{c}1,2\end{tabular}&\begin{tabular}{c}1\\2\end{tabular}\\\hline\end{tabular}\quad
\begin{tabular}{|c|c|}\hline\begin{tabular}{c}1,2\end{tabular}&\begin{tabular}{c}2\\1\end{tabular}\\\hline\end{tabular}\quad
\begin{tabular}{|c|c|}\hline\begin{tabular}{c}1\\2\end{tabular}&\begin{tabular}{c}1,2\end{tabular}\\\hline\end{tabular}\quad
\begin{tabular}{|c|c|}\hline\begin{tabular}{c}2\\1\end{tabular}&\begin{tabular}{c}1,2\end{tabular}\\\hline\end{tabular}
\]
There are $q_{2,1}=4$ pairs of weak orders on $\{1,2\}$ with 1 distinguished consecutive embedding:
\[
\begin{tabular}{|c|c|}\hline\begin{tabular}{c}1,2\end{tabular}&\begin{tabular}{c}1\\2\end{tabular}\\\hline\end{tabular}
\begin{tikzpicture}[overlay]
 \draw[red, line width=1pt] (-0.5,0.1) ellipse (0.2cm and 0.4cm);
\end{tikzpicture}\quad
\begin{tabular}{|c|c|}\hline\begin{tabular}{c}1,2\end{tabular}&\begin{tabular}{c}2\\1\end{tabular}\\\hline\end{tabular}
\begin{tikzpicture}[overlay]
 \draw[red, line width=1pt] (-0.5,0.1) ellipse (0.2cm and 0.4cm);
\end{tikzpicture}\quad
\begin{tabular}{|c|c|}\hline\begin{tabular}{c}1\\2\end{tabular}&\begin{tabular}{c}1,2\end{tabular}\\\hline\end{tabular}
\begin{tikzpicture}[overlay]
 \draw[red, line width=1pt] (-1.8,0.1) ellipse (0.2cm and 0.4cm);
\end{tikzpicture}\quad
\begin{tabular}{|c|c|}\hline\begin{tabular}{c}2\\1\end{tabular}&\begin{tabular}{c}1,2\end{tabular}\\\hline\end{tabular}
\begin{tikzpicture}[overlay]
 \draw[red, line width=1pt] (-1.8,0.1) ellipse (0.2cm and 0.4cm);
\end{tikzpicture}
\]
\end{example}
\begin{example}
Figure 2 depicts a more complicated example which we will return to later.
The example in Figure 2 is a pair of weak orders on $\{1,2,3,4,5,6,7,8,9\}$ with 5 total consecutive embeddings, 4 of which are distinguished:
\begin{figure}[H]
    \centering
    \begin{tabular}{|c|c|}\hline\begin{tabular}{c}2\\7\\5\\1,3,4,6,8,9\end{tabular}&\begin{tabular}{c}8\\6\\1,4\\3\\2,5,7\\9\end{tabular}\\\hline\end{tabular}
    \begin{tikzpicture}[overlay]
    \draw[red, line width=1pt] (-2.87,0.55) ellipse (0.2cm and 0.4cm);
    \draw[red, line width=1pt] (-2.87,0.12) ellipse (0.2cm and 0.4cm);
    \draw[red, line width=1pt] (-0.9,0.13) ellipse (0.4cm and 0.4cm);
    \draw[red, line width=1pt] (-0.9,0.96) ellipse (0.2cm and 0.4cm);
    \end{tikzpicture}
    \caption{A pair of weak orders with distinguished consecutive embeddings.}
\end{figure}
\end{example}
The first few values of $q_{n,k}$ are given in Table 1.
There are two main observations from this table that we will prove below.
First, $q_{n,k}=0$ for $k>n$ (Corollary \ref{cor:qnk=0}).
Second, $\sum_{k=0}^n(-1)^kq_{n,k}=q_n$ (Lemma \ref{lem:inclusion-exclusion}).
\begin{table}[H]
    \centering
    \begin{tabular}{c||c||c|c|c|c|c}
        $n$&$q_n$&$q_{n,0}$&$q_{n,1}$&$q_{n,2}$&$q_{n,3}$&$q_{n,4}$\\\hline\hline
        $0$&1&1&0&0&0&0\\
        $1$&1&1&0&0&0&0\\
        $2$&5&9&4&0&0&0\\
        $3$&97&169&84&12&0&0\\
        $4$&3365&5625&2812&600&48&0\\
        $5$&177601&292681&145380&34380&4320&240
    \end{tabular}
    \caption{Small Values of $q_n$ and $q_{n,k}$}
    \label{table:qnk}
\end{table}
We will need the notion of a chain in a pair of weak orders with distinguished consecutive embeddings.
\begin{definition}
In a pair of weak orders on $\{1,\ldots,n\}$ with distinguished consecutive embeddings, a \textit{chain} is a maximal sequence of overlapping distinguished consecutive embeddings.
\end{definition}
The example in Figure 2 has one chain on the left side and two chains on the right side.
The following lemma is a key property of chains.
\begin{lemma}
\label{lem:chainoverlap}
In a pair of weak orders on $\{1,\ldots,n\}$ with distinguished consecutive embeddings, any two chains have no elements of $\{1,\ldots,n\}$ in common, even if the two chains are on different sides.
\end{lemma}
\begin{proof}
It is clear that two chains on the same side have no elements of $\{1,\ldots,n\}$ in common.
It remains to show that any two consecutive embeddings on different sides have no elements of $\{1,\ldots,n\}$ in common.
Let $(i,i+1)_l$ be a left consecutive embedding and let $(i^\prime,i^\prime+1)_r$ be a right consecutive embedding.
Then $S_i\cup S_{i+1}\subseteq T_j$ for some (necessarily unique) $T_j$ and $T_{i^\prime}\cup T_{i^\prime+1}\subseteq S_{j^\prime}$ for some (necessarily unique) $S_{j^\prime}$.
If these two consecutive embeddings have an element of $\{1,\ldots,n\}$ in common, then $T_j$ must be either $T_{i^\prime}$ or $T_{i^\prime+1}$, and $S_{j^\prime}$ must be either $S_i$ or $S_{i+1}$.
Then
\[\abs{S_i\cup S_{i+1}}\leq\abs{T_j}<\abs{T_{i^\prime}\cup T_{i^\prime+1}}\leq\abs{S_{j^\prime}}<\abs{S_i\cup S_{i+1}}\]
which is a contradiction.
\end{proof}
We can now determine $q_{n,k}$ for $k\geq n-1$.
\begin{proposition}
\label{prop:inequality}
If $k\geq n$ and $k\geq1$, then $q_{n,k}=0$.
If $k=n-1$ and $k\geq1$, then $q_{n,k}=2\cdot n!$.
\end{proposition}
\begin{proof}
Consider a pair of weak orders on $\{1,\ldots,n\}$ with $k\geq1$ distinguished consecutive embeddings.
Let $c\geq1$ be the number of chains.
Note that $c+k$ counts the number of parts that are contained in the chains, because each chain has one more part than distinguished consecutive embedding.
Note that these $c+k$ parts have no elements in common, since any two parts within a chain have no elements in common, and Lemma \ref{lem:chainoverlap} states that distinct chains have no elements in common.
This gives the inequality $c+k\leq n$.

If $k\geq n$, then this is impossible, which shows that $q_{n,k}=0$.
Now suppose that $k=n-1$.
Then $c=1$, so there is one chain consisting of $n-1$ overlapping distinguished consecutive embeddings.
Thus, one side (either left or right) has $n$ parts in one of the $n!$ possible permutations, and the other side has just one part consisting of the unordered set $\{1,\ldots,n\}$.
This gives $2\cdot n!$ possibilities.
\end{proof}
\begin{corollary}
\label{cor:qnk=0}
If $k>n$, then $q_{n,k}=0$.
\end{corollary}
\begin{proof}
If $k>n$, then $k\geq1$ so Proposition \ref{prop:inequality} gives $q_{n,k}=0$.
\end{proof}
We now give the inclusion-exclusion formula for $q_n$ in terms of $q_{n,k}$.
\begin{lemma}
\label{lem:inclusion-exclusion}
We have the inclusion-exclusion formula
\[q_n=\sum_{k=0}^n(-1)^kq_{n,k}.\]
\end{lemma}
\begin{proof}
If a pair of weak orders on $\{1,\ldots,n\}$ has $m$ total consecutive embeddings, then $m\leq n$ by Corollary \ref{cor:qnk=0}, so that pair of weak orders is counted $\sum_{k=0}^m(-1)^k\binom{m}{k}$ times by the sum.
This is equal to 0 if $m\geq1$ and equal to 1 if $m=0$.
Thus, the sum counts pairs of weak orders with no consecutive embeddings.
\end{proof}
\subsection{The Final Count}
It remains to count pairs of weak orders on $\{1,\ldots,n\}$ with $k$ distinguished consecutive embeddings.
To deal with the fact that consecutive embeddings can overlap, we will work with chains.
The first step of our formula will be to sum over the number of chains, which we will call $c$, and the total number of elements from $\{1,\ldots,n\}$ to put inside these chains, which we will call $t$.
Keep in mind Lemma \ref{lem:chainoverlap}, which states that distinct chains have no elements in common, even if they are on different sides.
In other words, each of the $t$ elements will end up in exactly one of the $c$ chains.

There are $\binom{n}{t}$ ways to choose which $t$ elements from $\{1,\ldots,n\}$ to put inside the $c$ chains.
In the proof of Proposition \ref{prop:inequality} we saw that the $c$ chains contain $c+k$ total parts, so $t\geq c+k$.
Then there are $(c+k)!\Stwo{t}{c+k}$ ways to partition these $t$ elements into an ordered sequence of $c+k$ parts.

Now we cut the ordered sequence of $c+k$ parts into $c$ chains, each of which must consist at least 2 parts.
Since the parts are already ordered, this just requires choosing how many parts to put into each chain.
The number of ways to do this equals the number of ways to distribute $c+k$ identical balls among $c$ distinguishable urns, each of which must get at least 2 balls.
This is the same as the number of ways to surjectively distribute $k$ identical balls among $c$ distinguishable urns.
There are $\binom{k-1}{k-c}$ such ways.

At this point, we have an ordered sequence of $c$ chains, along with $n-t$ non-chain elements.
We can minimize the distinction between the $c$ chains and the $n-t$ elements by collapsing each chain into a single element, as depicted in Figure 3.
For example, consider the top right chain in Figure 3.
Both the left and right side of this chain get collapsed to a single element (which is labeled $c_2$ in Figure 3).
The key difference between the left $c_2$ and the right $c_2$ is that the right $c_2$ must be in a part by itself.
Thus, the circling in the collapsed form indicates which side of the collapsed chain must be in a part by itself.
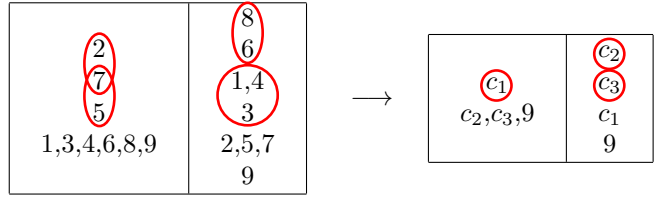
\begin{figure}[H]
\centering
\begin{tabular}{|c|c|}\hline\begin{tabular}{c}2\\7\\5\\1,3,4,6,8,9\end{tabular}&\begin{tabular}{c}8\\6\\1,4\\3\\2,5,7\\9\end{tabular}\\\hline\end{tabular}
\begin{tikzpicture}[overlay]
\draw[red, line width=1pt] (-2.87,0.55) ellipse (0.2cm and 0.4cm);
\draw[red, line width=1pt] (-2.87,0.12) ellipse (0.2cm and 0.4cm);
\draw[red, line width=1pt] (-0.9,0.13) ellipse (0.4cm and 0.4cm);
\draw[red, line width=1pt] (-0.9,0.96) ellipse (0.2cm and 0.4cm);
\end{tikzpicture}
$\quad\longrightarrow\quad$
\begin{tabular}{|c|c|}\hline\begin{tabular}{c}$c_1$\\$c_2$,$c_3$,9\end{tabular}&\begin{tabular}{c}$c_2$\\$c_3$\\$c_1$\\9\end{tabular}\\\hline\end{tabular}
\begin{tikzpicture}[overlay]
\draw[red, line width=1pt] (-2.2,0.26) ellipse (0.2cm and 0.2cm);
\draw[red, line width=1pt] (-0.7,0.68) ellipse (0.2
cm and 0.2cm);
\draw[red, line width=1pt] (-0.7,0.26) ellipse (0.2cm and 0.2cm);
\end{tikzpicture}
\caption{Collapsing chains.}
\end{figure}
We will need a generalization of the Fubini numbers that allows for the possibility of elements that are required to be placed in a part by themselves.
For $0\leq k\leq n$, let $f_{n,k}$ count the number of weak orders on $\{1,\ldots,n\}$ where each of the elements $1,\ldots,k$ must be placed in a part by itself.
These generalized Fubini numbers have the formula $f_{n,k}=\sum_{l=k}^nl!\Stwo{n-k}{l-k}$, because there are $\Stwo{n-k}{l-k}$ ways to partition the set $\{k+1,\ldots,n\}$ into $l-k$ parts, and $l!$ ways to order these $l-k$ parts along with the remaining $k$ elements $\{1,\ldots,k\}$.
The ordinary Fubini numbers are the special case when $k=0$.

Returning to the situation prior to Figure 3, we now sum over the number of chains on the left side, which we will call $j$.
This forces the number of chains on the right side to be $c-j$.
We already have an ordering on the $c$ chains, so we will take the first $j$ chains to be the chains on the left side, and the last $c-j$ chains to be the chains on the right side.
After collapsing each of the $c$ chains into a single element on each side, we have $n-t+c$ total elements to distribute.
On the left side, $j$ of these elements are marked as having to be placed in a part by themselves.
On the right side, $c-j$ of these elements are marked as having to be placed in a part by themselves.
Then there are $f_{n-t+c,j}f_{n-t+c,c-j}$ ways to produce a valid pair of weak orders.
However, since we already have an ordering for the $c$ chains, we must divide by $j!$ and $(c-j)!$.
This gives the formula
\[q_{n,k}=\sum_{c=0}^k\sum_{t=c+k}^n\underbrace{\binom{n}{t}}_{\substack{\text{choose $t$}\\\text{elements}}}\underbrace{(c+k)!\Stwo{t}{c+k}}_{\substack{\text{partition the $t$}\\\text{elements into an}\\\text{ordered sequence}\\\text{of $c+k$ parts}}}\underbrace{\binom{k-1}{k-c}}_{\substack{\text{split the}\\\text{$c+k$ parts}\\\text{into $c$ chains}}}\sum_{j=0}^c\f{\overbrace{f_{n-t+c,j}f_{n-t+c,c-j}}^{\substack{\text{collapsing each chain into}\\\text{a single element on each side}\\\text{makes $n$ become $n-t+c$}}}}{\underbrace{j!\,(c-j)!}_{\substack{\text{the chains are already}\\\text{ordered at this point,}\\\text{so the numerator is}\\\text{overcounting by this factor}}}}.\]
We should point out that the second sum is zero for $c>n-k$.
Plugging this into Lemma \ref{lem:inclusion-exclusion} and rearranging the summations gives the formula
\begin{equation}\label{eq:qn}q_n=\sum_{c=0}^{\lfloor n/2\rfloor}\sum_{t=2c}^n\binom{n}{t}\pq{\sum_{k=c}^{t-c}(-1)^k(c+k)!\Stwo{t}{c+k}\binom{k-1}{k-c}}\pq{\sum_{j=0}^c\frac{f_{n-t+c,j}f_{n-t+c,c-j}}{j!\,(c-j)!}}.\end{equation}
This proves our formula for $p_n$ (equation \ref{eq:pnform} from the introduction).
\begin{theorem}
\label{thm:pn_formula}
For each nonnegative integer $n$, we have the formula
\[p_n=\frac{1}{n!}\sum_{m=0}^n\Sone{n}{m}\sum_{c=0}^{\lfloor m/2\rfloor}\sum_{t=2c}^m\binom{m}{t}\pq{\sum_{k=c}^{t-c}(-1)^k(c+k)!\Stwo{t}{c+k}\binom{k-1}{k-c}}\pq{\sum_{j=0}^c\frac{f_{m-t+c,j}f_{m-t+c,c-j}}{j!\,(c-j)!}},\]
where the values $f_{n,k}=\sum_{l=k}^nl!\Stwo{n-k}{l-k}$ are generalizations of the Fubini numbers.
\end{theorem}
Equation \ref{eq:qn} has been written so that there are several independent components inside the formula.
The next subsection shows how we can precompute these components to speed up computation.
\subsection{The Algorithm}
We now give an algorithm to compute the values $p_1,\ldots,p_N$ in $O(N^2)$ memory and $O(N^3)$ time.
The statement ``$O(N^2)$ memory and $O(N^3)$ time'' assumes that storing a number requires constant memory and that arithmetic operations require constant time.
This assumption is not true in practice since the computation involves numbers that grow super-exponentially.
However, the number of digits of these numbers grows polynomially.
Then storing a number requires polynomial memory and arithmetic operations require polynomial time.
Thus, even with these practical considerations, the algorithm uses polynomial memory and polynomial time.
Empirically, the runtime for the algorithm grows like $N^{4.75}$.

Here is the algorithm:
\begin{enumerate}
    \item Compute the values $\displaystyle\binom{n}{k},\Sone{n}{k},\Stwo{n}{k}$, for $0\leq k\leq n\leq N$ using the standard recurrences.
    \item Compute the values $f_{n,k}=\displaystyle\sum_{l=k}^nl!\Stwo{n-k}{l-k}$ for $0\leq k\leq n\leq N$.
    \item Compute the values $g_{n,c}=\displaystyle\sum_{j=0}^c\f{f_{n,j}f_{n,c-j}}{j!\,(c-j)!}$ for $0\leq c\leq n\leq N$.
    \item Compute the values $h_{t,c}=\displaystyle\sum_{k=c}^{t-c}(-1)^k(c+k)!\Stwo{t}{c+k}\binom{k-1}{k-c}$ for $0\leq2c\leq t\leq n$.
    
    In this calculation, note that $\binom{-1}{0}=1$ and also $\binom{k-1}{k}=0$ for $k\geq1$.
    We will give a faster and simpler algorithm for $h_{t,c}$ in the next subsection.
    
    \item Compute the values $q_n=\displaystyle\sum_{c=0}^{\lfloor n/2\rfloor}\sum_{t=2c}^n\binom{n}{t}h_{t,c}g_{n-t+c,c}$ for $0\leq n\leq N$.
    \item Compute the values $p_n=\displaystyle\frac{1}{n!}\sum_{k=0}^n\Sone{n}{k}q_k$ for $0\leq n\leq N$.
\end{enumerate}
I have written two versions of this algorithm in Java (see Web Resources at the end of this paper).

Version 1 is a straightforward single-threaded implementation of this algorithm, with the faster and simpler step 4 given in the next subsection.
Version 1 can compute the values $p_1,\ldots,p_{1000}$ in 40 minutes on a laptop with a 2.4 GHz processor using 2 GB of memory.
Version 1 is limited by memory.

Version 2 is a memory-optimized multi-threaded implementation of the algorithm that was used to compute the values $p_1,\ldots,p_{5000}$.
The computation took 9.4 days on a server with 20 CPU cores, each running at 2.4 GHz.
Version 2 achieves $O(N)$ memory and $O(N^3)$ time (with the same caveat as before).
It achieves $O(N)$ memory by not storing precomputed two-dimensional arrays and instead recomputing values on the fly.
Version 2 maintains the $O(N^3)$ time of Version 1 by rewriting formula (\ref{eq:qn}) for $q_n$ so that the outside sum is a sum over the values of $n-t+c$, which avoids expensive recomputation of the values $g_{n-t+c,c}$.
Version 2 is limited by time, rather than by memory.
\subsection{Combinatorial Interpretation of \texorpdfstring{$h_{t,c}$}{Lg}}
We conclude this section by giving a combinatorial interpretation of $h_{t,c}$.
This will give a fast and simple algorithm for $h_{t,c}$, as well as a bound on $h_{t,c}$ that will be used when proving asymptotic formulas.
\begin{proposition}
\label{prop:htc_combo}
For all integers $0\leq2c\leq t$, the quantity
$(-1)^ch_{t,c}/2^c$ counts the number of ordered set-partitions of $\{1,\ldots,t\}$ into $c$ parts of even cardinality.
In particular, if $t$ is odd, then $h_{t,c}=0$.
\end{proposition}
\begin{proof}
First, we rewrite $h_{t,c}$ as
\[h_{t,c}=(-1)^c\sum_{k=2c}^t(-1)^kk!\Stwo{t}{k}\binom{k-c-1}{k-2c}.\]
Recall that the quantity
\[(c+k)!\Stwo{t}{c+k}\binom{k-1}{k-c}\]
from equation (\ref{eq:qn}) counts the number of ways to choose an ordered set-partition of $\{1,\ldots,t\}$ into $c+k$ parts, which are further grouped into $c$ chains of at least two parts each.
Then the quantity
\[k!\Stwo{t}{k}\binom{k-c-1}{k-2c}\]
counts the number of ways to choose an ordered set-partition of $\{1,\ldots,t\}$ into $k$ parts, which are further grouped into $c$ consecutive blocks of at least two parts each.
Here's example with $c=3$, $k=7$, and $t=10$:
\[\begin{tabular}{c|c|c}\begin{tabular}{|c|c|}\hline 3&5\\\hline\end{tabular}&\begin{tabular}{|c|c|c|}\hline 4&8,10&2\\\hline\end{tabular}&\begin{tabular}{|c|c|}\hline 1,7,9&6\\\hline\end{tabular}\end{tabular}\]
Considering the $c$ blocks gives an ordered set-partition of $\{1,\ldots,t\}$ into $c$ parts.
If we let $\Omega$ denote the collection of all ordered set-partitions of $\{1,\ldots,t\}$ into $c$ parts, then we can write
\[k!\Stwo{t}{k}\binom{k-c-1}{k-2c}=\sum_{\pi\in\Omega}a_\pi,\]
where $a_\pi$ counts the number of ways that $\pi\in\Omega$ arises as the $c$ blocks of an ordered set-partition of $\{1,\ldots,t\}$ into $k$ parts, which are further grouped into $c$ consecutive blocks of at least two parts each.
If $\pi$ is the ordered set-partition $\{1,\ldots,t\}=\pi_1\cup\cdots\cup \pi_c$, then we can write $a_\pi$ as
\[a_\pi=\sum_{\substack{k=k_1+\cdots+k_c\\k_i\geq2}}\prod_{i=1}^ck_i!\Stwo{\abs{\pi_i}}{k_i}.\]
Thus,
\[k!\Stwo{t}{k}\binom{k-c-1}{k-2c}=\sum_{\pi\in\Omega}\sum_{\substack{k=k_1+\cdots+k_c\\k_i\geq2}}\prod_{i=1}^ck_i!\Stwo{\abs{\pi_i}}{k_i}\]
Putting this all together gives
\begin{align*}
    h_{t,c}&=(-1)^c\sum_{k=2c}^t(-1)^k\sum_{\pi\in\Omega}\sum_{\substack{k=k_1+\cdots+k_c\\k_i\geq2}}\prod_{i=1}^ck_i!\Stwo{\abs{\pi_i}}{k_i}\\
    &=(-1)^c\sum_{\pi\in\Omega}\sum_{k=2c}^t\sum_{\substack{k=k_1+\cdots+k_c\\k_i\geq2}}\prod_{i=1}^c(-1)^{k_i}k_i!\Stwo{\abs{\pi_i}}{k_i}\\
    &=(-1)^c\sum_{\pi\in\Omega}\sum_{\substack{k_1,\ldots,k_c\\2\leq k_i\leq\abs{\pi_i}}}\prod_{i=1}^c(-1)^{k_i}k_i!\Stwo{\abs{\pi_i}}{k_i}\\
    &=(-1)^c\sum_{\pi\in\Omega}\prod_{i=1}^c\sum_{k_i=2}^{\abs{\pi_i}}(-1)^{k_i}k_i!\Stwo{\abs{\pi_i}}{k_i}.
\end{align*}
We now invoke the identity
\[\sum_{k=2}^n(-1)^kk!\Stwo{n}{k}=\begin{cases}2&\text{if }n\text{ is even},\\0&\text{if }n\text{ is odd}\end{cases}\]
which follows from setting $x=-1$ in the identity
\[\sum_{k=0}^n\Stwo{n}{k}(x)_k=x^n.\]
where $(x)_k=x(x-1)\cdots(x-k+1)$ denotes the falling factorial.
Thus, if each $\pi_i$ has even cardinality, then $\pi$ contributes $(-1)^c2^c$ to $h_{t,c}$, and otherwise $\pi$ contributes nothing to $h_{t,c}$.
The result follows.
\end{proof}
\begin{corollary}
\label{cor:htc_estimate}
For all integers $0\leq2c\leq t$, we have the inequality $\abs{h_{t,c}}\leq2^cc^t$.
\end{corollary}
\begin{proof}
Proposition \ref{prop:htc_combo} states that $\abs{h_{t,c}}/2^c$ counts the number of ordered set-partitions of $\{1,\ldots,t\}$ into $c$ parts of even cardinality.
This is bounded above by $c^t$ (the number of functions $\{1,\ldots,t\}\to\{1,\ldots,c\}$).
\end{proof}
The combinatorial interpretation of $h_{t,c}$ given in Proposition \ref{prop:htc_combo} also gives a recurrence for $h_{t,c}$.
It will be convenient to introduce an auxiliary sequence.
\begin{definition}
We define integers $T(n,k)$ for $0\leq k\leq n$ by
\begin{itemize}
    \item $T(n,0)=0$ for $n\geq1$,
    \item $T(n,n)=1$ for $n\geq0$,
    \item $T(n,k)=T(n-1,k-1)+k^2T(n-1,k)$ for $1\leq k\leq n-1$.
\end{itemize}
The values $T(n,k)$ maybe found at \cite[A036969]{oeis}.
\end{definition}
\begin{proposition}
\label{prop:hcomb}
For all integers $0\leq k\leq n$, we have $h_{2n,k}=(-1)^k(2k)!\,T(n,k)$.
\end{proposition}
The significance of Proposition \ref{prop:hcomb} is that it reduces the computation of $h_{t,c}$ in step 4 of the algorithm from cubic time to quadratic time.
Let $S(n,k)=(-1)^kh_{2n,k}/2^k$, which by Proposition \ref{prop:htc_combo} counts the number of ordered set-partitions of $\{1,\ldots,2n\}$ into $k$ parts of even cardinality.
The proof of Proposition \ref{prop:hcomb} boils down to showing $S(n,k)$ satisfies the recurrence $S(n,k)=k^2S(n-1,k)+k(2k-1)S(n-1,k-1)$.
This recurrence is known \cite[A241171]{oeis}.
We include a proof of Proposition \ref{prop:hcomb} for completeness.
\begin{proof}[Proof of Proposition \ref{prop:hcomb}]
Let $S(n,k)$ be defined as above.
It suffices to show that $S(n,k)=(2k)!\,T(n,k)/2^k$, because then $h_{2n,k}=(-1)^k2^kS(n,k)=(-1)^k(2k)!\,T(n,k)$.
We need to prove the following three properties:
\begin{enumerate}
    \item $S(n,0)=0$ for $n\geq1$.
    \item $S(n,n)=(2n)!/2^n$ for $n\geq0$.
    \item $S(n,k)=k^2S(n-1,k)+k(2k-1)S(n-1,k-1)$ for $1\leq k\leq n-1$.
\end{enumerate}
The first property follows from the observation that if $n\geq1$, then there are no ordered set-partitions of $\{1,\ldots,2n\}$ into 0 parts.
For the second property, note that an ordered set-partition of $\{1,\ldots,2n\}$ into $n$ parts of even cardinality consists of pairing up $\{1,\ldots,2n\}$ into an ordered sequence of $n$ unordered pairs.
There are $(2n)!/2^n$ such pairings.
It remains to show the third property (the recurrence).
Consider the following two ways of constructing an ordered set-partition of $\{1,\ldots,2n\}$ into $k$ parts of even cardinality:
\begin{itemize}
    \item Start with an ordered set-partition $\{1,\ldots,2n-2\}=S_1\cup\cdots\cup S_k$ into $k$ parts of even cardinality, and choose two (possibly equal) indices $i,j\in\{1,\ldots,k\}$.
    Consider the smallest element of the union $S_i\cup S_j$ and flip its position (if it was in $S_i$, then move it to $S_j$, and vice versa).
    Finally, add $2n-1$ to $S_i$ and add $2n$ to $S_j$.
    \item Start with an ordered set-partition $\{1,\ldots,2n-2\}=S_1\cup\cdots\cup S_{k-1}$ into $k-1$ parts of even cardinality, and choose an index $k_0\in\{1,\ldots,k\}$.
    Inserting the empty set at position $k_0$ and relabeling gives $\{1,\ldots,2n-2\}=S_1\cup\cdots\cup S_k$, which would be an ordered set-partition except that $S_{k_0}$ is empty.
    Now choose two (possibly equal) indices $i,j\in\{1,\ldots,k\}$, at least one of which is equal to $k_0$.
    Consider the smallest element of the union $S_i\cup S_j$ and flip its position (if $i=j=k_0$, then do nothing).
    Finally, add $2n-1$ to $S_i$ and add $2n$ to $S_j$.
\end{itemize}
We claim that every ordered set-partition of $\{1,\ldots,2n\}$ into $k$ parts of even cardinality arises uniquely from one of these two constructions.
This is because the process can be reversed:
\begin{itemize}
    \item Start with an ordered set-partition of $\{1,\ldots,2n\}=S_1\cup\cdots\cup S_k$ into $k$ parts of even cardinality.
    Let $S_i$ be the set containing $2n-1$ and let $S_j$ be the set containing $2n$.
    Remove $2n-1$ from $S_i$ and remove $2n$ from $S_j$.
    Consider the smallest element of the union $S_i\cup S_j$ and flip its position (if $S_i=S_j=\varnothing$, then do nothing).
    If $S_i$ and $S_j$ are still nonempty, then we are in the first case.
    If either $S_i$ or $S_j$ is now empty, then we are in the second case.
\end{itemize}
    In the first case, there are $S(n-1,k)$ ways of choosing the initial ordered set-partition, and $k^2$ ways of choosing the indices $i,j$.
    In the second case, there are $S(n-1,k-1)$ ways of choosing the initial ordered set-partition, $k$ ways of choosing the index $k_0$, and $2k-1$ ways of choosing the indices $i,j$.
    This proves the recurrence $S(n,k)=k^2S(n-1,k)+k(2k-1)S(n-1,k-1)$.
    \end{proof}
\section{Asymptotics}
\subsection{Asymptotics of \texorpdfstring{$f_{n,k}$}{Lg}}
We first give a formula for the generalized Fubini numbers $f_{n,k}$ in terms of the ordinary Fubini numbers $f_n$.
\begin{proposition}
\label{prop:fubiniformula}
If $k\geq0$ and $n\geq0$ are integers, then
\[f_{n+k,k}=k!\sum_{n_0+\cdots+n_k=n}\frac{n!}{n_0!\cdots n_k!}f_{n_0}\cdots f_{n_k},\]
where each $n_i$ is a nonnegative integer.
\end{proposition}
\begin{proof}
Recall that $f_{n+k,k}$ counts the number of weak orders on $\{1,\ldots,n+k\}$ where each of the elements $1,\ldots,k$ must be placed in a part by itself.
We first choose the permutation of the elements $1,\ldots,k$.
This leaves $k+1$ regions between the elements $1,\ldots,k$ where we must place the remaining $n$ elements $k+1,\ldots,n+k$.
We then choose how many elements $n_i$ should go into each of the $k+1$ regions.
There are $\f{n!}{n_0!\cdots n_k!}$ ways to choose how to distribute the $n$ elements $k+1,\ldots,n+k$ into these $k+1$ regions.
Finally, there are $f_{n_i}$ choices for the weak order on the $n_i$ elements within each of the $k+1$ regions.
\end{proof}
We remark that Proposition \ref{prop:fubiniformula} gives a generating function identity
\[\sum_{n=0}^\infty\frac{f_{n+k,k}}{n!}x^n=k!\pq{\sum_{n=0}^\infty\frac{f_n}{n!}x^n}^{k+1}=\frac{k!}{(2-e^x)^{k+1}}.\]
It is possible to prove Lemma \ref{lem:fubini_growth} below using complex-analytic generating function techniques by looking at the poles of $k!/(2-e^x)^{k+1}$.
However, we will prove Lemma \ref{lem:fubini_growth} using more direct real-analytic techniques.
We will need the following technical lemma.
\begin{lemma}
\label{lem:tech2}
Let $\{a_n\}_{n=0}^\infty$ be a sequence of real numbers and let $k\geq0$ be an integer.
If $a_n\to1$, then
\[\frac{1}{\binom{n+k}{k}}\sum_{\substack{n_0+\cdots+n_k=n\\\text{all }n_i\geq0}}a_{n_0}\cdots a_{n_k}\to1\]
as $n\to\infty$.
\end{lemma}
\begin{proof}
We first remark that $\binom{n+k}{k}$ is the number of terms in the sum, since the number of solutions of $n_0+\cdots+n_k=n$ equals the number of ways to distribute $n$ identical balls into $k+1$ distinguishable urns.
Let $\varepsilon>0$.
Then there exists an $m\geq0$ such that $a_n\in[1-\varepsilon,1+\varepsilon]$ for $n\geq m$.
Let $C=\max\abs{a_n}$.
For $n\geq(k+1)m$, we will split up the sum as
\[\sum_{n_0+\cdots+n_k=n}a_{n_0}\cdots a_{n_k}=\sum_{\substack{n_0+\cdots+n_k=n\\\text{all }n_i\geq m}}a_{n_0}\cdots a_{n_k}+\sum_{\substack{n_0+\cdots+n_k=n\\\text{some }n_i<m}}a_{n_0}\cdots a_{n_k}.\]
The number of terms of the first sum equals the number of solutions to $n_0+\cdots+n_k=n$ with all $n_i\geq m$, which equals the number of solutions to $n_0+\cdots+n_k=n-(k+1)m$.
Then the first sum has $\binom{n-(k+1)m+k}{k}$ terms (by the same argument as at the start of the proof), each of which lies in the interval $[(1-\varepsilon)^{k+1},(1+\varepsilon)^{k+1}]$.
The second sum has $\binom{n+k}{k}-\binom{n-(k+1)m+k}{k}$ terms, each of which lies in the interval $[-C^{k+1},C^{k+1}]$.
Thus,
\[(1-\varepsilon)^{k+1}\binom{n-(k+1)m+k}{k}-C^{k+1}\pq{\binom{n+k}{k}-\binom{n-(k+1)m+k}{k}}\leq\sum_{n_0+\cdots+n_k=n}a_{n_0}\cdots a_{n_k}\]
and
\[\sum_{n_0+\cdots+n_k=n}a_{n_0}\cdots a_{n_k}\leq(1+\varepsilon)^{k+1}\binom{n-(k+1)m+k}{k}+C^{k+1}\pq{\binom{n+k}{k}-\binom{n-(k+1)m+k}{k}}.\]
Dividing through by $\binom{n+k}{k}$ gives the inequalities
\[(1-\varepsilon)^{k+1}\frac{\binom{n-(k+1)m+k}{k}}{\binom{n+k}{k}}-C^{k+1}\pq{1-\frac{\binom{n-(k+1)m+k}{k}}{\binom{n+k}{k}}}\leq\frac{1}{\binom{n+k}{k}}\sum_{n_0+\cdots+n_k=n}a_{n_0}\cdots a_{n_k}\]
and
\[\frac{1}{\binom{n+k}{k}}\sum_{n_0+\cdots+n_k=n}a_{n_0}\cdots a_{n_k}\leq(1+\varepsilon)^{k+1}\frac{\binom{n-(k+1)m+k}{k}}{\binom{n+k}{k}}+C^{k+1}\pq{1-\frac{\binom{n-(k+1)m+k}{k}}{\binom{n+k}{k}}}.\]
Note that
\[\frac{\binom{n-(k+1)m+k}{k}}{\binom{n+k}{k}}=\frac{(n-(k+1)m+1)\cdots(n-(k+1)m+k)}{(n+1)\cdots(n+k)}\to1\text{ as }n\to\infty.\]
Then taking $n\to\infty$ gives
\begin{align*}(1-\varepsilon)^{k+1}\leq\liminf_{n\to\infty}\frac{1}{\binom{n+k}{k}}\sum_{n_0+\cdots+n_k=n}a_{n_0}\cdots a_{n_k}\leq\limsup_{n\to\infty}\frac{1}{\binom{n+k}{k}}\sum_{n_0+\cdots+n_k=n}a_{n_0}\cdots a_{n_k}\leq(1+\varepsilon)^{k+1}.\end{align*}
The result follows from taking $\varepsilon\to0$.
\end{proof}
We can now give an asymptotic formula for $f_{n,k}$ as $n\to\infty$.
Recall that $f(n)\sim g(n)$ means $\lim\limits_{n\to\infty}\frac{f(n)}{g(n)}=1$.
\begin{lemma}
\label{lem:fubini_growth}
For each fixed nonnegative integer $k$, we have \[f_{n,k}\sim\frac{n!}{2^{k+1}(\log2)^{n+1}}\]
as $n\to\infty$.
\end{lemma}
\begin{proof}
Proposition \ref{prop:fubiniformula} gives the combinatorial identity
\[f_{n+k,k}=k!\sum_{n_0+\cdots+n_k=n}\frac{n!}{n_0!\cdots n_k!}f_{n_0}\cdots f_{n_k}\]
which we can rewrite as
\[\frac{f_{n+k,k}}{\displaystyle\pq{\frac{(n+k)!}{2^{k+1}(\log2)^{n+k+1}}}}=\frac{1}{\displaystyle\binom{n+k}{k}}\sum_{n_0+\cdots+n_k=n}\frac{f_{n_0}}{\displaystyle\pq{\frac{n_0!}{2(\log2)^{n_0+1}}}}\cdots\frac{f_{n_k}}{\displaystyle\pq{\frac{n_k!}{2(\log2)^{n_k+1}}}}.\]
Then the result follows from Lemma \ref{lem:tech2}, along with the asymptotic formula $f_n\,{\sim}\,\frac{n!}{2(\log2)^{n+1}}$ \cite[p.175-176]{wilf}.
\end{proof}
\subsection{Asymptotics of \texorpdfstring{$q_n$}{Lg}}
We will need a short technical lemma.
\begin{lemma}
\label{lem:tech1}
For all integers $0\leq2c\leq t\leq n$,
\[\binom{n}{t}\pq{\frac{(n-t+c)!}{n!}}^2\leq\frac{1}{t!}.\]
\end{lemma}
\begin{proof}
We have
\[\binom{n}{t}\pq{\frac{(n-t+c)!}{n!}}^2=\frac{1}{t!}\frac{n!}{(n-t)!}\pq{\frac{(n-t+c)!}{n!}}^2=\frac{1}{t!}\frac{(n-t+c)!^2}{(n-t)!\,n!},\]
where
\[\frac{(n-t+c)!^2}{(n-t)!\,n!}\leq\frac{(n-t+c)!}{(n-t)!}\frac{(n-t+c)!}{(n-t+2c)!}=\frac{n-t+1}{n-t+c+1}\cdots\frac{n-t+c}{n-t+2c}\leq1.\qedhere\]
\end{proof}
\begin{theorem}
\label{thm:qn_growth}
We have
\[q_n\sim e^{-(\log2)^2}q_{n,0}\sim e^{-(\log2)^2}\frac{n!^2}{4\,(\log2)^{2n+2}}.\]
Thus, approximately $e^{-(\log2)^2}$ \emph{(}around $62\%$\emph{)} of pairs of weak orders have no consecutive embeddings.
\end{theorem}
\begin{proof}
Taking the formula $q_n=\sum_c\sum_t\binom{n}{t}h_{t,c}g_{n-t+c,c}$ in equation (\ref{eq:qn}), unfolding the definition of $g_{n-t+c,c}$, and dividing through by $f_n^2$ gives
\[\frac{q_n}{f_n^2}=\sum_{c=0}^{\lfloor n/2\rfloor}\sum_{t=2c}^n\sum_{j=0}^c\binom{n}{t}\frac{h_{t,c}}{j!\,(c-j)!}\frac{f_{n-t+c,j}f_{n-t+c,c-j}}{f_n^2}.\]
Now fix $c\geq0$, $t\geq2c$, and $0\leq j\leq c$, and consider the summand
\[\binom{n}{t}\frac{h_{t,c}}{j!\,(c-j)!}\frac{f_{n-t+c,j}f_{n-t+c,c-j}}{f_n^2}\]
for $n\geq t$ as $n\to\infty$.
By the asymptotic formula for $f_{n,k}$ in Lemma \ref{lem:fubini_growth} and the fact that $f_n=f_{n,0}$ we have
\begin{align*}
\binom{n}{t}\frac{h_{t,c}}{j!\,(c-j)!}\frac{f_{n-t+c,j}f_{n-t+c,c-j}}{f_n^2}&\sim\frac{n^t}{t!}\frac{h_{t,c}}{j!\,(c-j)!}\frac{\displaystyle\pq{\frac{(n-t+c)!}{2^{j+1}(\log2)^{n-t+c+1}}}\pq{\frac{(n-t+c)!}{2^{c-j+1}(\log2)^{n-t+c+1}}}}{\displaystyle\pq{\frac{n!}{2(\log2)^{n+1}}}^2}\\
&\sim\frac{n^t}{t!}\frac{h_{t,c}}{j!\,(c-j)!}\frac{(\log2)^{2t-2c}}{2^cn^{2t-2c}}\sim\frac{h_{t,c}}{j!\,(c-j)!}\frac{(\log2)^{2t-2c}}{2^ct!}\frac{1}{n^{t-2c}}.
\end{align*}
If $t>2c$, then the summand converges to 0.
If $t=2c$, then $h_{t,c}=(-1)^ct!$ (this follows from the definition of $h_{t,c}$, as well as from Proposition \ref{prop:hcomb}) so the summand converges to
\[(-1)^c\frac{1}{j!\,(c-j)!}\frac{(\log2)^{2c}}{2^c}.\]
After applying dominated convergence theorem (justified below), the $t>2c$ terms contribute nothing so we can drop the sum over $t$ and get
\[\frac{q_n}{f_n^2}\to\sum_{c=0}^\infty\sum_{j=0}^c(-1)^c\frac{1}{j!\,(c-j)!}\frac{(\log2)^{2c}}{2^c}=\sum_{c=0}^\infty(-1)^c\frac{(\log2)^{2c}}{c!}=e^{-(\log2)^2}.\]
The result follows from the identity $q_{n,0}=f_n^2$ and the asymptotic formula $f_n\sim\frac{n!}{2(\log2)^{n+1}}$ (Lemma \ref{lem:fubini_growth}).

It remains to justify this application of the dominated convergence theorem.
Note that the asymptotic formula $f_n\sim\frac{n!}{2(\log2)^{n+1}}$ gives a constant $C$ (not depending on $c$, $t$, $j$) such that we have
\begin{align*}
    \left\lvert\binom{n}{t}\frac{h_{t,c}}{j!\,(c-j)!}\frac{f_{n-t+c,j}f_{n-t+c,c-j}}{f_n^2}\right\rvert&=\frac{\abs{h_{t,c}}}{j!\,(c-j)!}\binom{n}{t}\frac{f_{n-t+c,j}f_{n-t+c,c-j}}{f_n^2}\\
    &\leq\frac{\abs{h_{t,c}}}{j!\,(c-j)!}\binom{n}{t}\pq{\frac{f_{n-t+c}}{f_n}}^2\\
    &\leq C\frac{\abs{h_{t,c}}}{j!\,(c-j)!}\binom{n}{t}\pq{\frac{\displaystyle\frac{(n-t+c)!}{(\log2)^{n-t+c}}}{\displaystyle\frac{n!}{(\log2)^n}}}^2\\
    &\leq C\frac{\abs{h_{t,c}}}{j!\,(c-j)!}\binom{n}{t}\pq{\frac{(n-t+c)!}{n!}}^2\\
    &\leq C\frac{\abs{h_{t,c}}}{j!\,(c-j)!}\frac{1}{t!},
\end{align*}
where the last inequality uses Lemma \ref{lem:tech1}.
It remains to show that the sum
\[\sum_{c=0}^\infty\sum_{t=2c}^\infty\sum_{j=0}^c\frac{\abs{h_{t,c}}}{j!\,(c-j)!}\frac{1}{t!}\]
is finite.
Applying Corollary \ref{cor:htc_estimate} and the inequality $2c\leq t$ from the bounds of summation gives
\begin{align*}
    \sum_{c=0}^\infty\sum_{t=2c}^\infty\sum_{j=0}^c\frac{\abs{h_{t,c}}}{j!\,(c-j)!}\frac{1}{t!}&=\sum_{c=0}^\infty\sum_{t=2c}^\infty\frac{2^c}{c!\,t!}\abs{h_{t,c}}\leq\sum_{c=0}^\infty\sum_{t=2c}^\infty\frac{2^{2c}c^t}{c!\,t!}\leq\sum_{c=0}^\infty\sum_{t=2c}^\infty\frac{(2c)^t}{c!\,t!}\leq\sum_{c=0}^\infty\sum_{t=0}^\infty\frac{(2c)^t}{c!\,t!}=\sum_{c=0}^\infty\frac{e^{2c}}{c!}
\end{align*}
which is finite by the ratio test.
\end{proof}
\subsection{Asymptotics of \texorpdfstring{$p_n$}{Lg}}
We will need the following technical lemma.
\begin{lemma}
\label{lem:Sone_growth}
For each fixed nonnegative integer $k$, we have
\[\Sone{n}{n-k}\sim\frac{n^{2k}}{2^kk!}\]
as $n\to\infty$.
We also have
\[\Sone{n}{n-k}\leq\frac{n^{2k}}{2^kk!}\]
for all $n\geq k$.
\end{lemma}
\begin{proof}
The first statement follows from equation 1.6 of \cite{moser}.
For the second statement, we will use the recurrence for the Stirling numbers of the first kind.
If $k=0$ or $k=n$, then the inequality is clear.
Now suppose that $1\leq k\leq n-1$, and inductively assume that the inequality holds for smaller values of $n$.
Then
\begin{align*}
    \Sone{n}{n-k}&=(n-1)\Sone{n-1}{n-k}+\Sone{n-1}{n-k-1}\\
    &=(n-1)\Sone{n-1}{(n-1)-(k-1)}+\Sone{n-1}{(n-1)-k}\\
    &\leq(n-1)\frac{(n-1)^{2(k-1)}}{2^{k-1}(k-1)!}+\frac{(n-1)^{2k}}{2^kk!}\\
    &=\frac{1}{2^kk!}(2k+n-1)(n-1)^{2k-1}\\
    &\leq\frac{1}{2^kk!}\pq{\frac{(2k+n-1)+(2k-1)(n-1)}{2k}}^{2k}\\
    &=\frac{n^{2k}}{2^kk!},
\end{align*}
where the last inequality (the penultimate step) uses the AM-GM inequality.
\end{proof}
We can now prove the asymptotic formula for $p_n$ (Theorem \ref{thm:pnthm} from the introduction).
\begin{theorem}
\label{thm:pn_growth}
We have
\[p_n\sim e^{-(\log2)^2/2}\frac{f_n^2}{n!}\sim\frac{e^{-(\log2)^2/2}}{4(\log2)^2}\frac{n!}{(\log2)^{2n}}.\]
\end{theorem}
\begin{proof}
Applying the formula $\displaystyle p_n=\frac{1}{n!}\sum\limits_{k=0}^n\Sone{n}{k}q_k$ and replacing $k$ with $n-k$ gives
\[\frac{n!\,p_n}{f_n^2}=\sum_{k=0}^n\Sone{n}{k}\frac{q_k}{f_n^2}=\sum_{k=0}^n\Sone{n}{n-k}\frac{q_{n-k}}{f_n^2}.\]
If $k$ is fixed, then applying Lemma \ref{lem:fubini_growth}, Theorem \ref{thm:qn_growth}, and Lemma \ref{lem:Sone_growth} gives
\[\Sone{n}{n-k}\frac{q_{n-k}}{f_n^2}\sim\frac{n^{2k}}{2^kk!}\frac{\displaystyle e^{-(\log2)^2}\frac{(n-k)!^2}{4(\log2)^{2(n-k)+2}}}{\displaystyle\pq{\frac{n!}{2(\log2)^{n+1}}}^2}\sim\frac{e^{-(\log2)^2}(\log2)^{2k}}{2^kk!}.\]
By the dominated convergence theorem (justified below),
\[\frac{n!\,p_n}{f_n^2}\to\sum_{k=0}^\infty\frac{e^{-(\log2)^2}(\log2)^{2k}}{2^kk!}=e^{-(\log2)^2}e^{(\log2)^2/2}=e^{-(\log2)^2/2}.\]
Then the result follows from the asymptotic formula $f_n\sim\frac{n!}{2(\log2)^{n+1}}$.

It remains to justify this application of the dominated convergence theorem.
Note that the asymptotics $f_n\sim\frac{n!}{2(\log2)^{n+1}}$ and $q_n\sim e^{-(\log2)^2}\frac{n!}{4(\log2)^{2n+2}}$ give a constant $C$ (not depending on $n$ or $k$) such that we have
\[\Sone{n}{n-k}\frac{q_{n-k}}{f_n^2}\leq C\frac{n^{2k}}{2^kk!}\frac{\displaystyle\frac{(n-k)!^2}{(\log2)^{2(n-k)}}}{\displaystyle\pq{\frac{n!}{(\log2)^n}}^2}\leq C\frac{n^{2k}}{2^kk!}\frac{(n-k)!^2}{n!^2}\]
for $n\geq k$.
Now view the expression
\[n^{2k}\frac{(n-k)!^2}{n!^2}=\pq{\pq{\frac{n}{n-k+1}}\pq{\frac{n}{n-k+2}}\cdots\pq{\frac{n}{n}}}^2\]
as a function of $n\geq k$ for fixed $k$.
This function has a fixed number of terms, each of which is decreasing as a function of $n$.
In particular, this function is maximized at $n=k$ with value $k^{2k}/k!^2$.
This shows that
\[\Sone{n}{n-k}\frac{q_{n-k}}{f_n^2}\leq C\frac{n^{2k}}{2^kk!}\frac{(n-k)!^2}{n!^2}\leq C\frac{k^{2k}}{2^kk!^3}\]
for $n\geq k$.
Finally, the sum
\[\sum_{k=0}^\infty\frac{k^{2k}}{2^kk!^3}\]
is finite by the ratio test.
\end{proof}
\section{Further Directions}
\subsection{Higher Order Asymptotics}
Let $K$ denote the constant
\[K=\frac{e^{-(\log2)^2/2}}{4(\log2)^2}\approx0.409223.\]
Theorem \ref{thm:pn_growth} states that
\[\frac{p_n(\log2)^{2n}}{n!}\to K.\]
Figure 4 is a graph of $y=\log\big(K-\frac{p_n(\log2)^{2n}}{n!}\big)$ against $x=\log n$ for the values of $n$ given in the appendix.
\begin{figure}[H]
\centering
\includegraphics[scale=0.5]{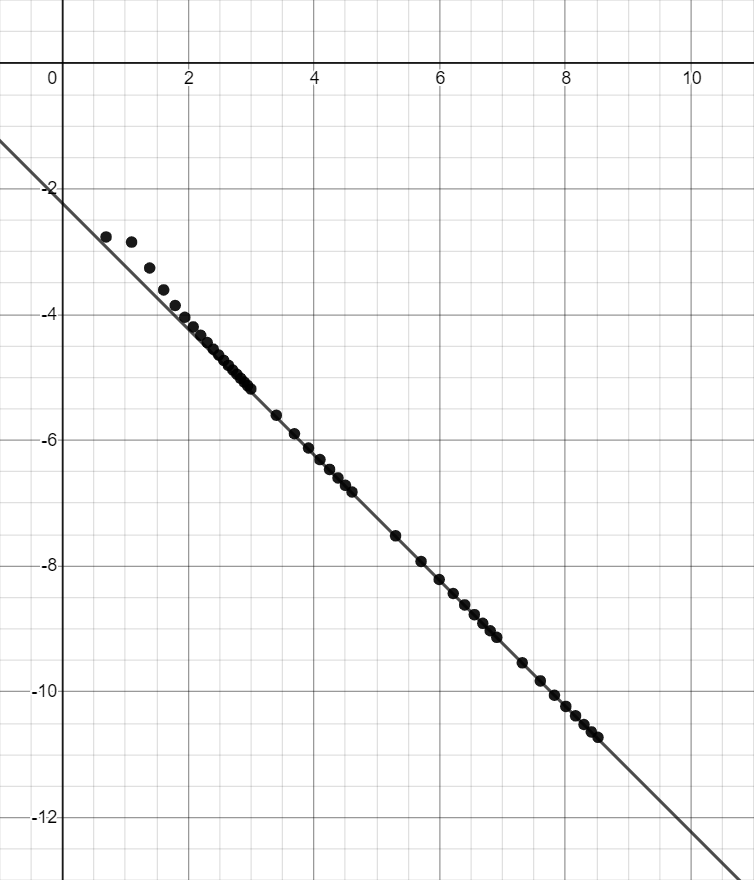}
\caption{Graph of $y=\log\big(K-\frac{p_n(\log2)^{2n}}{n!}\big)$ against $x=\log n$}
\end{figure}
\noindent
For large $n$, the points appear to approach a line with slope $-1$ and $y$-intercept $b\approx-2.23$.
In other words,
\[\log\pq{K-\frac{p_n(\log2)^{2n}}{n!}}\approx b-\log n.\]
Exponentiating and rearranging terms gives
\[\frac{p_n(\log2)^{2n}}{n!}\approx K-\frac{e^b}{n}.\]
This suggests the following conjecture.
\begin{conjecture}
There exists a constant $c>0$ such that
\[\frac{p_n(\log2)^{2n}}{n!}=K-\frac{c}{n}+O\pq{\frac{1}{n^2}}.\]
\end{conjecture}
\subsection{Congruence Conjecture}
Theorem 2 of \cite{poonen} states that if $p$ is prime and $n\geq m$, then
\[f_{n+\varphi(p^m)}\equiv f_n\pmod{p^m},\]
where $\varphi$ is Euler's totient function.
In other words, the Fubini numbers are periodic modulo $p^m$.
Squaring both sides gives the congruence
\[q_{n+\varphi(p^m),0}\equiv q_{n,0}\pmod{p^m}\]
for $p$ prime and $n\geq m$.
Replacing $q_{n,0}$ with $q_n$ suggests the following conjecture, which is supported by the computed values of $q_n$ for $n\leq5000$.
\begin{conjecture}
If $p$ is prime and $n\geq m$, then
\[q_{n+\varphi(p^m)}\equiv q_n\pmod{p^m}.\]
\end{conjecture}
Unfortunately, this conjecture does not appear to imply anything about the behavior of the sequence $\{p_n\}$ modulo $p^m$, because of the division by $n!$ when converting from $\{q_n\}$ to $\{p_n\}$.
\section{Acknowledgements}
Many thanks to Sara Billey for proposing the problem and for helpful discussions, and to the WXML (Washington Experimental Mathematics Lab) for providing the project that led to this paper.
Also thanks to Sara Billey and the anonymous referees for their comments and suggestions on this paper.
\section{Web Resources}
The basic and memory-optimized implementations of the algorithm described in section 3.6 and the values of $p_n$ for $n\leq5000$ can be found at \[\text{https://github.com/tb65536/ParabolicDoubleCosets}\]
\section{Appendix: Values of \texorpdfstring{$p_n$}{Lg}}
\[\begin{tabular}{r|l r|l}
$n$&$p_n$&&$(p_n\log^{2n}2)/n!$\\\hline
1&1&(1 digits)&0.480453\\
2&3&(1 digits)&0.346253\\
3&19&(2 digits)&0.3512\\
4&167&(3 digits)&0.370774\\
5&1791&(4 digits)&0.382093\\
6&22715&(5 digits)&0.388048\\
7&334031&(6 digits)&0.391663\\
8&5597524&(7 digits)&0.394169\\
9&105351108&(9 digits)&0.396036\\
10&2200768698&(10 digits)&0.397485\\
11&50533675542&(11 digits)&0.398644\\
12&1265155704413&(13 digits)&0.399593\\
13&34300156146805&(14 digits)&0.400385\\
14&1001152439025205&(16 digits)&0.401056\\
15&31301382564128969&(17 digits)&0.401631\\
16&1043692244938401836&(19 digits)&0.402131\\
17&36969440518414369896&(20 digits)&0.402569\\
18&1386377072447199902576&(22 digits)&0.402955\\
19&54872494774746771827248&(23 digits)&0.403299\\
20&2285943548113541477123970&(25 digits)&0.403608\\
30&382079126820...882950534546&(42 digits)&0.405528\\
40&179736290098...532574927537&(61 digits)&0.406469\\
50&102365379120...338473199289&(81 digits)&0.407028\\
60&427699505826...027450945465&(101 digits)&0.407398\\
70&940027093836...926979570377&(122 digits)&0.407662\\
80&857360695445...439742054481&(144 digits)&0.407859\\
90&271659624624...300501685746&(167 digits)&0.408011\\
100&260443549181...383464403196&(190 digits)&0.408133\\
200&150691150471...390138470043&(439 digits)&0.40868\\
300&400039289653...047576602840&(710 digits)&0.408862\\
400&572423854465...686938545249&(996 digits)&0.408952\\
500&745894661762...526127432358&(1293 digits)&0.409006\\
600&529056570650...070570426529&(1599 digits)&0.409042\\
700&692359539273...658799872850&(1912 digits)&0.409068\\
800&150717237472...313160125048&(2232 digits)&0.409088\\
900&902565318506...968550812571&(2556 digits)&0.409103\\
1000&367762337807...336792083803&(2886 digits)&0.409115\\
1500&657393993927...489306609387&(4592 digits)&0.409151\\
2000&677187561025...781759174668&(6372 digits)&0.409169\\
2500&497164609537...894142980291&(8207 digits)&0.40918\\
3000&189293873430...434167136044&(10086 digits)&0.409187\\
3500&163043229993...353705274487&(12001 digits)&0.409192\\
4000&186384435725...985721119395&(13947 digits)&0.409196\\
4500&346443781530...440739293425&(15920 digits)&0.409199\\
5000&962766473267...951984139754&(17917 digits)&0.409201\\
\end{tabular}\]

\end{document}